\documentclass[12pt, leqno]{article}
\usepackage{amsmath, amssymb, mathtools, amsthm}
\usepackage[numbers]{natbib}
\usepackage{bm}
\usepackage[inline]{enumitem}
\usepackage[colorlinks = true, linkcolor = blue, citecolor = blue, pdfpagelabels]{hyperref}
\usepackage[dvipsnames]{xcolor}
\usepackage{framed}
\usepackage{soul}

\makeatletter
\def\NAT@spacechar{~}
\makeatother

\colorlet{shadecolor}{yellow}
\definecolor{light-gray}{gray}{0.95}

\theoremstyle{plain}
\newtheorem{theorem}{Theorem}

\newtheorem{proposition}{Proposition}

\newtheorem{lemma}{Lemma}

\theoremstyle{remark}
\newtheorem{remark}{Remark}

\theoremstyle{definition}
\newtheorem{definition}{Definition}

\newtheorem*{definition*}{Definition}

\newtheorem{assumption}{Assumption}

\allowdisplaybreaks

\newcommand\lp{\left (}
\newcommand\rp{\right )}

\newcommand\var{\hbox{\rm Var}}

\def\convd{\stackrel{\mbox{$\scriptstyle d$}}{\rightarrow}}

\DeclareMathOperator*{\argmin}{arg\,min}

\newcommand \eg{\varepsilon}
\newcommand \la{\lambda}
\newcommand \og{\omega}

\begin{document}

\title{Estimation of the long-run variance of nonlinear
time series with an application to change point analysis}

\author{Vaidotas Characiejus\footnote{Correspondence to:
Vaidotas Characiejus,  Department of Mathematics and Computer Science,
University of Southern Denmark, Campusvej 55, 5230 Odense M, Denmark.\newline Email: characiejus@imada.sdu.dk}\\
{\small University of Southern Denmark}
\and
Piotr Kokoszka \\
{\small Colorado State University}
\and
Xiangdong Meng \\
{\small Colorado State University}
}
\date{\today}
\maketitle

\begin{abstract}

For a broad class of nonlinear time series known as Bernoulli
shifts, we establish the asymptotic normality of the smoothed
periodogram estimator of the long-run variance. This estimator
uses only a narrow band of Fourier frequencies around the origin
and so has been extensively used in local Whittle estimation.
Existing asymptotic normality results
apply only to linear time series, so  our work substantially
extends the scope of the applicability of the  smoothed
periodogram estimator. As an illustration, we
apply it to a test of changes in mean against long-range dependence.
A simulation study is also conducted to illustrate the
performance of the test for nonlinear time series.

\bigskip


\bigskip


\bigskip

\end{abstract}

\section{Introduction} \label{s:int}
Over  the past two decades, classical temporal dependence conditions
based on
cumulants or mixing coefficients have been complemented
by conditions formulated in
terms of nonlinear moving averages of  the form
\begin{equation}\label{eq:defX_t}
	X_t = g(\varepsilon_t,\varepsilon_{t-1},\ldots), \ \ \ t\in\mathbb Z,
\end{equation}
where $g$ can be a fairly general function.
Among them,  the {\em Physical or Functional Dependence}
and the {\em Approximability}
have gained particular popularity and have been used in dozens
of publications. Physical dependence was introduced by
\citet{wu:2005b} and studied or assumed in
\citet{Wu2010},  \citet{Liu:probability:2013},
\citet{zhou:2013},
\citet{zhang2017asymptotic}, \citet{dette:kokot:volgushev:2020} and
\citet{vandelft:2020}, among many others.
Approximability was also  used in a large number of papers, for example, in
\citet{aue:hormann:horvath:reimherr:2009},
\citet{hormann:kokoszka:2010}, \citet{berkes:horvath:rice:2013},
\citet{horvath:kokoszka:reeder:2013},
\citet{horvath:kokoszka:rice:2014},
\citet{kokoszka:miao:zhang:2015}, \citet{zhang:2016},
\citet{bardsley:2017} and \citet{horvath:kokoszka:lu:2024}.

Estimation of the spectral density function,  which includes
estimation of the long-run variance (LRV),  is a key problem in time
series analysis that has been studied for several decades.
The LRV of a scalar time series is introduced in most
time series textbooks, e.g.\ in \citet{hamilton:1994}.
Multivariate, high-dimensional and functional analogs of the long-run
variance have been
extensively investigated. The  contributions of \citet{wu:xiao:2012},
\citet{berkes:horvath:rice:2016},   \citet{chan:2022}
and \citet{baek:duker:pipiras:2023}, among many others,
provide many references. \citet{mcelroy:politis:2024} explain
why estimation of the spectral density at frequencies near zero
requires care and propose effective solutions.

Specialized to the estimation of the LRV,  the lag window estimator,
\begin{equation}\label{eq:lag-w-general}
	\hat{f}(0) = \frac{1}{2\pi} \sum_{|r| \le n-1}w_n(r)
\hat{\gamma}(r),
\end{equation}
where the $w_n(r)$ are weights such that $w_n(r)\to0$,
 as $|r|\to\infty$,  and the $\hat{\gamma}(r)$ 
 are the usual sample autocovariances,
has been particularly extensively studied.
\citet{anderson:1971} developed its theory for linear processes,
while \citet{rosenblatt:1984} considered strong mixing processes
satisfying cumulant summability conditions.
These results were extended to nonlinear time series.
\citet{chanda:2005} assumed  a specific form
of the general Bernoulli shifts~\eqref{eq:defX_t},
namely (for a fixed $w$),
$
X_t = \eg_t + \sum_{r=1}^\infty g_r (\eg_{t-1}, \ldots,  \eg_{t-r-w}),
$
with summability conditions on the fourth moments of
the  $g_r (\eg_{t-1}, \ldots,  \eg_{t-r-w})$. The above representation
is motivated by a truncated Volterra expansion.
Results of \citet{shao:wu:2007b}, \citet{wu:shao:2007},
and \citet{liu:wu:2010}, which  assume only
the general representation \eqref{eq:defX_t},
can be specialized to yield  asymptotic normality
of the estimator \eqref{eq:lag-w-general}
under the additional assumption that the lag 
windows in~\eqref{eq:lag-w-general} 
are of the scale parameter form (i.e., $w_n(r)=w(r/B_n)$, 
where $w$ is a kernel and $B_n$ is a bandwidth) and
under various conditions
related to  the physical dependence of \citet{wu:2005b}.

In this paper, we consider the smoothed periodogram estimator
\begin{equation}\label{e:Qn}
Q_n = \frac1m\sum_{j=1}^m I_n(\omega_j),
\end{equation}
where $I_n(\omega_j)$ is the periodogram at the Fourier frequency
$\omega_j$ and $m$ is the count of frequencies used in the estimation
($m/n \to 0$). While the lag window estimator \eqref{eq:lag-w-general}
can be expressed as a weighted sum of the periodogram over {\em all}
Fourier frequencies, estimator~\eqref{e:Qn} uses only a narrow band
of local Fourier frequencies. For this reason, it is particularly
relevant in local Whittle estimation, see  e.g.\
\citet{robinson:1995}, \citet{velasco:1999},
\citet{phillips:shimotsu:2004}, \citet{robinson:2008},
\citet{baek:pipiras:2012},
\citet{li:robinson:shang:2021} and  \citet{baek:kokoszka:meng:2024}.

To the best of our knowledge, the asymptotic distribution of 
estimator~\eqref{e:Qn}
is only available under the assumption of linearity,
see \citet{giraitis:koul:surgailis:2012} for a 
comprehensive account. The estimator~\eqref{e:Qn} 
can be alternatively expressed as the lag window 
estimator~\eqref{eq:lag-w-general} with $w_n(0)=1$ and
\begin{equation}\label{eq:Qn_w}
	w_n(r)
	=\frac1{2m}\Bigl[\frac{\sin((2m+1)\pi r/n)}{\sin(\pi r/n)}-1\Bigr]
\end{equation}
for $0<|r|<n-1$. The $w_n(r)$ in~\eqref{eq:Qn_w} 
are  not of the scale parameter form and the asymptotic 
distribution of $Q_n$ does not follow from the currently 
available results for nonlinear time series.

A related but a different estimator of the LRV is Daniell's estimator given by (see, for example, Chapter~9 of \citet{anderson:1971})
\begin{equation}\label{eq:daniell}
	\frac{n/m}\pi\int_0^{\pi/(n/m)}I_n(\omega)\, d\omega
	=\frac1{2\pi}\sum_{r=-(T-1)}^{T-1}\frac{\sin(\pi r/(n/m))}{\pi r/(n/m)}\hat\gamma(r)
\end{equation}
with the convention that $\sin(\pi r/(n/m))/(\pi r/(n/m))=1$ 
when $r=0$. Daniell's lag window $w(r)=\sin(\pi r)/(\pi r)$ 
is of the scale parameter form so the asymptotic normality 
of Daniell's estimator for nonlinear time series follows 
directly from the currently available results 
(see Theorem~2 of \citet{liu:wu:2010}),
 unlike the asymptotic normality of the estimator $Q_n$.

In this paper, we establish the asymptotic normality of $Q_n$ for general nonlinear
moving averages and present an application to  change point analysis.
We note that \citet{brillinger:1969} considers an estimator
similar to \eqref{e:Qn} (see his equation (6.6)), but for a
fixed $m$, and \eqref{e:Qn} is not even consistent then.
Our proofs are based on a general central limit 
theorem for quadratic forms established by \citet{liu:wu:2010},  
which we state as Theorem~\ref{thm:liu:wu:2010:6}.

The relationship between the lag window estimator at any frequency $\og$
and the periodogram ordinates is explained in Section 6.2.3 of
\citet{priestley:1981} and Section 3 of \citet{shao:wu:2007b},
among others. To emphasize the difference, we note that
the lag window estimator can be expressed
as $\hat{f}(\og) = \int_{-\pi}^\pi K_n(\og - \la) I_n(\la) d\la$,
so it involves all frequencies. On the other hand,  the quadratic
form of the lag window estimator involves only the terms $X_t X_\tau$ with
$|t-\tau|\le \ell_n$, whereas the smoothed periodogram involves
these terms with all indexes $t, \tau$. These differences explain
why different techniques of proof must be used for these two
classes of estimators.

The paper is organized as follows. In Section \ref{s:pre}, we introduce
the objects we study. Section \ref{s:est} contains  the main theorems,
while Sections \ref{s:cpa} and \ref{s:sim}
present, respectively, a theoretical application to change point analysis
and a related,  small simulation study.
All proofs are collected in Section~\ref{s:p},  
which also contains a central limit theorem 
for weighted sums of the periodogram ordinates that 
we use in our proofs,  and that is also of independent interest. 
Graphs and Tables related to the simulation study are
presented in online Supporting Information.

\section{Preliminaries} \label{s:pre}
\subsection{The spectral density and the periodogram} \label{ss:sd}
We assume in the
following that $\{X_t\}_{t\in\mathbb Z}$  is a strictly stationary
sequence of real random variables satisfying
\[
\operatorname EX_1 = 0\quad\text{and}\quad\operatorname EX_1^2 < \infty.
\]
Let $\{\gamma(h)\}_{h\in\mathbb Z}$ be the sequence of autocovariances
of $\{X_t\}_{t\in\mathbb Z}$, where
$\gamma(h)=\operatorname{Cov}(X_h,X_0)$ for $h\in\mathbb Z$.
The spectral density function of the series $\{X_t \}_{t\in\mathbb Z}$
is defined by
\begin{equation}\label{eq:spectraldensity}
	f(\omega) = \frac1{2\pi}\sum_{h=-\infty}^\infty
\gamma(h)e^{-ih\omega}, \ \ \ \omega\in[-\pi,\pi].
\end{equation}
The discrete Fourier transform (DFT) of $X_1,\ldots,X_n$
and the periodogram are defined by
\begin{equation}\label{eq:dftandper}
	\mathcal X_n(\omega_j)
	=\frac1{\sqrt{2\pi n}}\sum_{t=1}^nX_te^{-it\omega_j} \ \ \
\text{and} \ \ \
	I_n(\omega_j) = |\mathcal X_n(\omega_j)|^2
\end{equation}
with the Fourier frequencies $\omega_j=2\pi j/n$,
where $j\in\{-\lfloor (n-1)/2\rfloor,\ldots,\lfloor n/2\rfloor\}$ and $\lfloor\cdot\rfloor$ is the floor function.

\subsection{Weak dependence}
Let $\{X_t\}_{t\in\mathbb Z}$ be a sequence of
real random variables of the form \eqref{eq:defX_t},
where $\{\varepsilon_t\}_{t\in\mathbb Z}$ is a sequence of independent and
identically distributed (iid) random elements in a measurable space
$S$ and $g:S^\infty\to\mathbb R$ is a measurable function.
Series of the form \eqref{eq:defX_t} are often called Bernoulli shifts
and are automatically strictly stationary.

Suppose that $\{\varepsilon_t'\}_{t\in\mathbb Z}$ is an independent copy
of $\{\varepsilon_t\}_{t\in\mathbb Z}$. Define $X_{t,\{0\}}$
by replacing $\varepsilon_0$ in~\eqref{eq:defX_t}
by $\varepsilon_0'$, i.e.,
\begin{equation}\label{eq:Xt0}
	X_{t,\{0\}}
	=  g(\varepsilon_t,\ldots,\varepsilon_1,\varepsilon_0',\varepsilon_{-1},
\ldots), \ \ \ t\in\mathbb Z.
\end{equation}
We write $X\in\mathcal L^p$ with $p>0$ if $\|X\|_p
= (\operatorname E|X|^p)^{1/p}<\infty$.
The physical dependence measure is defined by
\begin{equation}\label{eq:delta}
	\delta_{t,p}
	= \|X_t-X_{t,\{0\}}\|_p, \ \ \ t\in\mathbb Z.
\end{equation}
Observe that $\delta_{t,p}=0$ if $t<0$.
Even though $\delta_{t,p}$ can be considered for any positive
value of $p$, typically $p\ge1$ is assumed
since otherwise $\|\cdot\|_p$ is no longer a norm
(we will work with $p=4$).
For completeness, we recall  Definition~3 of \citet{wu:2005b}.

\begin{definition} The series
$\{X_t\}_{t\in\mathbb Z}$ of the form \eqref{eq:defX_t}  is
$p$-strong stable if
\begin{equation}\label{eq:Thetamp}
	\Theta_p
	= \sum_{t=0}^\infty\delta_{t,p}<\infty, \ \ \ p\ge1,
\end{equation}
with  the $\delta_{t,p}$ defined in \eqref{eq:delta} and \eqref{eq:Xt0}.
\end{definition}

The decay rate of the physical dependence measure $\delta_{t,2}$ is related to the decay rate of the autocovariances $\gamma(h)$. For $r\ge0$,
\begin{equation}\label{eq:covbound}
	\sum_{h=1}^\infty h^r|\gamma(h)|
\le\sum_{t=0}^\infty\delta_{t,2}\sum_{h=1}^\infty h^r\delta_{h,2}
\end{equation}
since $|\gamma(h)|\le\sum_{t=0}^\infty\delta_{t,2}\delta_{t+h,2}$ 
(see Lemma~8 of \citet{xiao:wu:2012}) so that $\sum_{h=1}^\infty h^r|\gamma(h)|<\infty$ if $\sum_{h=1}^\infty h^r\delta_{h,2}<\infty$. In particular, $\Theta_2<\infty$ implies absolute summability of the autocovariances. The decay rate of the physical dependence measure is also related to the summability of joint cumulants up to certain orders (see Theorem~4.1 of \citet{shao:wu:2007a}).

\citet{wu:2005b} also considers the case when instead of
replacing $\varepsilon_0$ with its independent copy,
$\varepsilon_t$'s in~\eqref{eq:defX_t} with indices
in some index set $I\subset\mathbb Z$ are replaced with their
independent copies. If all but the first $m$
$\varepsilon_t$'s in~\eqref{eq:defX_t}
are replaced with their independent copies, we obtain the concept
of $L^p$-$m$-approximability. Again for completeness, we recall
Definition~2.1 of \citet{hoermann:kokoszka:2010}.

\begin{definition} The series
$\{X_t\}_{t\in\mathbb Z}$ of the form \eqref{eq:defX_t}  is
$L^p$-$m$-approximable if
\begin{equation}\label{eq:Lpmdefinition}
	\sum_{t=1}^\infty\|X_t-X_t^{(t)}\|_p<\infty, \ \ \ p\ge1,
\end{equation}
where
\begin{equation}\label{eq:Xtm}
	X_t^{(m)}
	=g(\varepsilon_t,\ldots,
\varepsilon_{t-m+1},\varepsilon_{t-m}',\varepsilon_{t-m-1}',\ldots).
\end{equation}
\end{definition}
Observe that $\|X_t-X_t^{(t)}\|_p
=\|X_s-X_s^{(t)}\|_p$ for $s,t\in\mathbb Z$ so that the right-hand  side
of~\eqref{eq:Lpmdefinition} can be formulated solely in terms
of $X_1$ and the approximations $X_1^{(t)}$ with $t\ge1$.

The following lemma shows  that $L^p$-$m$-approximability
implies $p$-strong stability.
\begin{lemma}\label{l:comp}
The inequality
\begin{equation}\label{eq:pstabLpm}
	\delta_{t,p}=\|X_t-X_{t,\{0\}}\|_p
	\le2\|X_t-X_t^{(t)}\|_p
\end{equation}
holds for $t\ge1$ and $p\ge1$ with $X_t$, $X_{t,\{0\}}$, and $X_t^{(t)}$
given by~\eqref{eq:defX_t}, \eqref{eq:Xt0}, and \eqref{eq:Xtm},  respectively.
\end{lemma}

\begin{proof}[Proof of \autoref{l:comp}]
The proof essentially
follows from the proof of Lemma~3.0$^\star$ of \citet{vandelft:2019}
(see page~5 therein), but we give it here for the sake of convenience
(the published version, \citet{vandelft:2020},  does not contain
this lemma).

Denote the $\sigma$-algebras generated by the random elements
$\varepsilon_t,\varepsilon_{t-1},\ldots$ with $t\in\mathbb Z$
and $\varepsilon_t,\ldots,\varepsilon_s$ with $t\ge s$ by
$\mathcal F_t=\sigma(\varepsilon_t,\varepsilon_{t-1},\ldots)$
and $\mathcal F_{t,s}=\sigma(\varepsilon_t,\ldots,\varepsilon_s)$,
respectively. We have that
\[
	\|X_t-X_{t,\{0\}}\|_p
	\le\|X_t-\operatorname E[X_t\mid\mathcal F_{t,1}]\|_p
	+\|\operatorname E[X_t\mid\mathcal F_{t,1}]-X_{t,\{0\}}\|_p.
\]
Since $\{\varepsilon_t':t\in\mathbb Z\}$ is an independent copy of $\{\varepsilon_t:t\in\mathbb Z\}$,
\[
	\operatorname E[X_t\mid\mathcal F_{t,1}]
	=\operatorname E[X_t^{(t)}\mid\mathcal F_{t,1}]
	=\operatorname E[X_t^{(t)}\mid\mathcal F_t]
\]
almost surely and hence
\begin{equation}\label{eq:bounddepass}
	\|X_t-\operatorname E[X_t\mid\mathcal F_{t,1}]\|_p
	=\|\operatorname E[X_t\mid\mathcal F_t]-\operatorname E[X_t^{(t)}\mid\mathcal F_t]\|_p
	\le\|X_t-X_t^{(t)}\|_p
\end{equation}
using Jensen's inequality for conditional expectations with $p\ge1$. We conclude the proof by noticing that $\operatorname E[X_t\mid\mathcal F_{t,1}]$ does not depend on $\varepsilon_0$ and hence
\[
	\|\operatorname E[X_t\mid\mathcal F_{t,1}]-X_{t,\{0\}}\|_p
	=\|\operatorname E[X_t\mid\mathcal F_{t,1}]-X_t\|_p,
\]
which is bounded in the same way as in~\eqref{eq:bounddepass}.
The proof is complete.
\end{proof}

\begin{remark} \label{r:stronger}
The $L^p$-$m$-approximability is in fact strictly
stronger than $p$-strong stability.
Consider, for example, a linear process $\{X_t\}_{t\in\mathbb Z}$, where $X_t=\sum_{j=0}^\infty\psi_j\varepsilon_{t-j}$ for $t\in\mathbb Z$, $\{\varepsilon_t\}_{t\in\mathbb Z}$ is a sequence of iid random variables such that $\operatorname E\varepsilon_0=0$ and $\operatorname E\varepsilon_0^2<\infty$, while $\{\psi_j\}_{j\ge0}$ is a sequence such that $\sum_{j=0}^\infty\psi_j^2<\infty$. Observe that
\[
	\sum_{t=0}^\infty\|X_t-X_{t,\{0\}}\|_2
	=2^{1/2}\|\varepsilon_0\|_2\sum_{t=0}^\infty|\psi_t|
	\quad\text{and}\quad
	\sum_{t=1}^\infty\|X_t-X_t^{(t)}\|_2
	 =2^{1/2}\|\varepsilon_0\|_2\sum_{t=1}^\infty\Bigl(\sum_{j=t}^\infty\psi_j^2\Bigr)^{1/2}.
\]
Suppose that $\psi_j=(j+1)^{-d}$ for $j\ge0$ with some $d>1$. Then
\[
	\Bigl(\sum_{j=t}^\infty\psi_j^2\Bigr)^{1/2}
	\sim\frac{t^{1/2-d}}{(2d-1)^{1/2}}
	\quad\text{as}\quad t\to\infty.
\]
So $\{X_t\}_{t\in\mathbb Z}$ is $2$-strong
stable for any value of $d>1$ but $\{X_t\}_{t\in\mathbb Z}$
is $L^2$-$m$-approximable only if $d>3/2$.
\end{remark}

\section{Estimation of the long-run variance} \label{s:est}
Consider the time series $\{ X_t \}_{t\in\mathbb Z}$ and its spectral density $f$, both
defined in Section \ref{ss:sd}. The long-run variance of $\{ X_t \}_{t\in\mathbb Z}$
is defined as
\[
\tau_X : = \sum_{k= -\infty}^\infty \gamma(h) = 2\pi f(0).
\]
The estimation of $\tau_X$ can thus be reduced to the estimation of
$f(0)$, and this approach is commonly used.
A natural  estimator of $f(0)$, known as the smoothed periodogram or the
nonparametric  spectral estimator given by \eqref{e:Qn}.
The chief contribution of this
paper is showing  that  the estimator $Q_n$ is asymptotically
normal, with the rate $m^{-1/2}$,  for a broad class of nonlinear
moving averages of the form \eqref{eq:defX_t}.
We work under the following general assumption.

 \begin{assumption} \label{a:eta} The series $\{ X_t \}_{t\in\mathbb Z}$ admits
representation \eqref{eq:defX_t}  and for some $\eta>3$,
\begin{equation}\label{e:HRW2.7}
	\|X_t-X_t^{(m)}\|_4=O(m^{-\eta})\quad\text{as}\quad m\to\infty.
\end{equation}
\end{assumption}

It follows from \autoref{l:comp} that the bound in \autoref{a:eta} holds for the physical dependence measure $\delta_{t,4}$. Using inequality~\eqref{eq:covbound} and the fact that $\delta_{t,2}\le\delta_{t,4}$, we see that \autoref{a:eta} implies that
\begin{equation} \label{e:h3g}
\sum_{h=1}^\infty h^r|\gamma(h)|<\infty, \ \ \ 0\le r < \eta -1,
\end{equation}
which in turn implies that $f(0)$ is well-defined under \autoref{a:eta} due to the absolute summability of the autocovariances.

We now  state our main result.

\begin{theorem}\label{thm:res1}
Assume that $\operatorname EX_0=0$,
$\operatorname E|X_0|^4<\infty$, \autoref{a:eta} holds  and $m=m_n\to\infty$, but $m=o(n^{4/5})$,
as $n\to\infty$. Then
\begin{equation}\label{e:res1}
	\frac{\sqrt m}{f(0)}[Q_n-f(0)]
	=\frac1{\sqrt m}\sum_{j=1}^m\biggl(\frac{I_n(\omega_j)}{f(0)}-1\biggr)
\xrightarrow d\mathcal N(0,1), \ \ \ {\rm as} \  n \to\infty,
\end{equation}
where $Q_n$ is defined in~\eqref{e:Qn} and $f(0)$
in~\eqref{eq:spectraldensity}.
\end{theorem}

\begin{remark}
The smoothed periodogram estimator typically  has the form
\[
\hat{f}(\og_k) = \frac{1}{m+1} \sum_{j=-m/2}^{m/2}
I_n\lp \og_k + \frac{2\pi j}{n}\rp.
\]
Thus, for $\og_k$  away from zero $\var[\hat{f}(\og_k)] \approx
f(\og_k)^2/m$, whereas at $k = 0$,
$\var[\hat{f}(\og_k)] \approx 2f(\og_k)^2/m$
because $I_n(\og_j) = I_n(-\og_j)$.
The estimator $Q_n$ does not include the $\og_{-j}$, so the factor 2 does
not appear.
\end{remark}

\autoref{thm:res1} follows from
Theorems~\ref{t:clt1} and \ref{thm:convrate}  formulated below.
We decompose the sequence in~\eqref{e:res1} into two terms
\begin{equation}\label{eq:twoparts}
	\frac{\sqrt m}{f(0)}[Q_n-f(0)]
	=\frac{\sqrt m}{f(0)}(Q_n-\operatorname EQ_n)+\frac{\sqrt{m}}{f(0)}(\operatorname EQ_n-f(0)).
\end{equation}
\autoref{t:clt1}, which implies that the first term on the right-hand
side of~\eqref{eq:twoparts} converges to $\mathcal N(0,1)$ in distribution,
is established under much weaker assumptions than \autoref{thm:res1}.
In particular,  the $L^p$-$m$-approximability of \autoref{a:eta}
is replaced by the weaker $p$-strong stability (by \autoref{l:comp},
\autoref{a:eta} implies that
$\Theta_4<\infty$). Furthermore, only the weakest possible assumption
on the rate of $m$ is imposed in \autoref{t:clt1} (\autoref{a:sublinear} below). We need to impose
a stronger assumption on the growth rate of $m$ as well as the decay rate of the covariances in \autoref{thm:res1}
than in \autoref{t:clt1} to ensure that the second term on the right-hand
side of~\eqref{eq:twoparts} is asymptotically negligible. \autoref{a:eta}
implies that $\sum_{h=1}^\infty h^2 |\gamma(h)| <\infty$  and hence,
according to \autoref{thm:convrate},
$\operatorname EQ_n-f(0)=O(\max\{n^{-1},(m/n)^2\})$, as $n\to\infty$.
The assumption in \autoref{thm:res1} that $m=o(n^{4/5})$ as $n\to\infty$ then
implies that $\sqrt m(m/n)^2=(m/n^{4/5})^{5/2}=o(1)$, as $n\to\infty$. 
If $\eta>1$ in \autoref{a:eta}, the autocovariances 
are absolutely summable but this only implies that 
$\operatorname EQ_n-f(0)=o(1)$ as $n\to\infty$ 
(see \autoref{thm:convrate} below),
 which is not sufficient to replace $\operatorname EQ_n$ 
 with $f(0)$ as the centering sequence in the central limit theorem.

Observe that~\eqref{e:res1} implies that
\begin{equation} \label{e:res2}
	\frac1m\sum_{j=1}^m\biggl(\frac{I_n(\omega_j)}{f(0)}-1\biggr)
	=O_p(m^{-1/2}), \ \ \ {\rm as} \ n\to\infty.
\end{equation}
Both \eqref{e:res1} and \eqref{e:res2} are used in theoretical work
related to the estimation of the long-run variance. We give an application
in Section \ref{s:cpa}.

For ease  of reference,
we  formulate the standard assumption on the growth rate of $m$.

\begin{assumption}\label{a:sublinear}
As $n\to\infty$, $m\to\infty$,  but $m=o(n)$.
\end{assumption}

\begin{theorem}\label{t:clt1}
Assume that $\operatorname EX_0=0$, $\operatorname E|X_0|^4<\infty$, $\Theta_4<\infty$
with $\Theta_4$ defined by~\eqref{eq:Thetamp}, and \autoref{a:sublinear} holds. Then
\begin{equation}\label{eq:clt1}
	\frac{\sqrt{m}}{f(0)}(Q_n-\operatorname EQ_n)
	\xrightarrow d\mathcal N(0,1), \ \ \ {\rm as} \ n\to\infty,
\end{equation}
where $Q_n$ is defined in~\eqref{e:Qn} and $f(0)$
in~\eqref{eq:spectraldensity}.
\end{theorem}
The proof of~\autoref{t:clt1} is given in Section~\ref{s:p} which follows from a general central limit theorem for weighted sums of periodogram ordinates which in turn
relies on a general central limit theorem for quadratic forms
established by~\citet{liu:wu:2010} (see Theorem~6 therein,
\citet{wu:shao:2007} consider a similar estimator).

Next, we investigate the bias of the estimator $Q_n$. The bias of a general lag window estimator~\eqref{eq:lag-w-general} depends on the curvature properties of the spectral density function and the particular lag windows. The classical results (see Theorem~9.3.3 of \citet{anderson:1971} and Chapter~6 of \citet{priestley:1981}) establish the decay rate of the bias only under the additional assumption that the lag windows are of the scale parameter and use the concept of the characteristic exponent introduced by \citet{parzen:1957}. Since the lag windows of the estimator $Q_n$ are not of the scale parameter form,
we establish the decay rate of the bias of $Q_n$ in the next theorem (the proof of~\autoref{thm:convrate} is given in Section~\ref{s:p}).

\begin{theorem}\label{thm:convrate}
Suppose that $\operatorname EX_0=0$, $\operatorname EX_0^2<\infty$, $\sum_{h=1}^\infty h^r|\gamma(h)|<\infty$ with some $r\ge0$ and \autoref{a:sublinear} holds. Then
\[
	\operatorname EQ_n-f(0)
	=
	\begin{cases}
	o(1)&\text{if}\ r=0,\\
	O((m/n)^r)&\text{if}\ 0<r\le 1,\\
	O(\max\{n^{-1},(m/n)^{\min\{2,r\}}\})&\text{if}\ r>1,
	\end{cases}
\]
where $Q_n$ is defined in~\eqref{e:Qn} and $f(0)$
in~\eqref{eq:spectraldensity}.
\end{theorem}
Recall that the decay rate of $\delta_{t,2}$ implies a decay rate of the autocovariances. In particular, if $\sum_{h=1}^\infty h^r\delta_{h,2}<\infty$, then $\sum_{h=1}^\infty h^r|\gamma(h)|<\infty$ (see inequality~\eqref{eq:covbound}).

\section{Application to change point analysis} \label{s:cpa}
A problem that has attracted some attention over the past
twenty years is distinguishing between long memory, or long-range
dependence (LRD), and changes in mean. Stationary time series used to model
LRD  exhibit
spurious changes in mean that  Benoit Mandelbrot called  the Joseph effect,
referring to periods of famine and plenty. This is a fundamental
characteristic of all stationary long memory models, see e.g.\
\citet{beranLM:2013}. \citet{berkes:horvath:kokoszka:shao:2006} proposed
a significance test of  the null hypothesis of changes in mean against
a long memory alternative. They review earlier exploratory
research on this problem. Under the null hypothesis, the observations
$X_t$ satisfy
\begin{equation}\label{e:CM}
    X_t = \begin{cases}
        \mu + r_t, & 1 \leq t \leq n^{\star},\\
        \mu + \Delta + r_t, & n^{\star}+1 \leq t \leq N,
    \end{cases}
\end{equation}
where $\mu$ is the base mean level, $\Delta$ the
magnitude of change at an unknown change point
$n^{\star}$ and $\{r_t\}_{t\in\mathbb Z}$ is a  stationary,
weakly dependent series.
\citet{berkes:horvath:kokoszka:shao:2006} characterized the
weak dependence of the $r_t$ by the convergence of their  partial
sums to the Wiener process and summability conditions on the
autocovariances and fourth order cumulants.
\citet{baek:pipiras:2012} pointed out that the  test of
\citet{berkes:horvath:kokoszka:shao:2006} can have
low power and proposed a more powerful test of \eqref{e:CM} based on the
local Whittle estimator (LWE) of the Hurst parameter, $H$. However,
they could justify their procedure only for {\em linear} processes
$\{ r_t \}_{t\in\mathbb Z}$. They relied on  results of \citet{robinson:1995}
valid only for linear time series.
We show in this section that the approach of
\citet{baek:pipiras:2012} is applicable also to nonlinear processes satisfying
the assumptions of \autoref{thm:res1}.
All results of this section are proven in Section \ref{s:p4}.

We begin with the definitions of the Hurst exponent  $H$.

\begin{definition} \label{d:H}Suppose $\{X_t\}_{t\in\mathbb Z}$ is a second order
stationary time series. If its spectral density satisfies
\[
f(\og) = |\og|^{1-2H}g(\og), \ \ \ {\rm for \ some} \ \ 0<H < 1,
\]
where $g$ is continuous on $[-\pi, \pi]$ with $g(\og) > 0$ in a neighborhood
of $\og =0$,  then
$H$ is called the Hurst exponent of the series $\{X_t\}$.
\end{definition}

Our objective is to test

\smallskip

\noindent $H_0$: Model \eqref{e:CM} holds with a series $\{ r_t\}$
satisfying Definition \ref{d:H} with $H=1/2$.

\smallskip

vs.

\smallskip

\noindent $H_A$: The series $\{X_t\}$ satisfies
Definition \ref{d:H} with $H>1/2$.

\smallskip

The test is based on the following estimator.
More background can be found in Chapter~8 of \citet{giraitis:koul:surgailis:2012}.

\begin{definition} \label{def:LWE}
The LWE of the Hurst exponent  $H$  is defined as
\begin{equation} \label{e:bp42}
    \widehat{H} = \argmin_{H \in \Theta} L(H),
\end{equation}
where $\Theta = [\Delta_1, \Delta_2]$
with $0 < \Delta_1 < \Delta_2 < 1$ and
\begin{equation} \label{e:bp43}
L(H) = \log \left(\frac{1}{m}
\sum_{l=1}^m \omega_l^{2H-1} I_{n}(\omega_l) \right) - (2H-1)
\frac{1}{m} \sum_{l=1}^m \log \omega_l.
\end{equation}
\end{definition}

As noted above, currently,  the validity of the test
of \citet{baek:pipiras:2012} is established  only
if the series $\{r_t\}$ in the null hypothesis is a linear process.
We show that the test retains correct asymptotic size  if the  $\{r_t\}$ is
a Bernoulli shift satisfying \autoref{a:eta}.

To construct the test statistic, the first step is to eliminate
the effect of a  potential change point that can
severely bias the  LWE. Consider the residuals
\[
R_t = \begin{cases}
    X_t - \frac{1}{\hat{n}} \sum_{i=1}^{\hat{n}}X_i, & 1 \leq t \leq \hat{n},\\
    X_t - \frac{1}{N - \hat{n}} \sum_{i=\hat{n}+1}^{N}X_i, &\hat{n} + 1 \leq t \leq N,
\end{cases}
\]
where $\hat{n}$ is a  change point estimator. We use the standard
CUSUM estimator
\[
\hat{n} = \min \left\{k: \left| \sum_{j=1}^k X_j
- \frac{k}{N}\sum_{j=1}^N X_j \right| = \max_{1\leq s \leq N}
\left| \sum_{j=1}^s X_j - \frac{s}{N}\sum_{j=1}^N X_j \right| \right\}.
\]
The test statistic is
\begin{equation} \label{e:TS}
    T^{(R)} = 2\sqrt{m} \lp \widehat H^{(R)} - \frac{1}{2} \rp,
\end{equation}
where $\widehat H^{(R)}$ is the LWE based on the residual
process $\{R_t\}$. We will show that  $T^{(R)}$ is asymptotically
standard normal.
The null hypothesis is the rejected
if $T^{(R)} > q_{1-\alpha}$, where $q_{1-\alpha}$ is the $(1-\alpha)$th
quantile of the standard normal distribution.

\begin{theorem} \label{t:TS}
Suppose model  \eqref{e:CM} with the $r_t$
satisfying  $\operatorname Er_0=0$, $\operatorname E|r_0|^4<\infty$
and  \autoref{a:eta}.
In addition, assume that \\
   (i) $n^{\star} = \lfloor N\theta \rfloor$ for some $\theta \in (0,1)$; \\
   (ii) the change in mean level $\Delta = \Delta(N)$ and the change
   point estimator $\hat{n}$ satisfy
    \[
    N\Delta^2 \to \infty, \quad \Delta^2\left| \hat{n} - n^{\star}\right|
    = O_P(1),\quad {\rm as } \ N \to \infty.
    \]
If
\[
\frac{1}{m} + \frac{m^{5} (\log m)^2}{N^4}
+ \frac{m (\log m)^2}{N \Delta^2} \to 0,\quad {\rm as } \ N \to \infty,
\]
then $T^{(R)} \convd {\cal N} (0,1)$.
\end{theorem}

The proof of \autoref{t:TS} follows the lines of  the proof  Theorem 3,
of \citet{baek:pipiras:2012} as long as the following result holds
and is applied to $X_t=r_t$.

\begin{theorem} \label{t:LWE}
Suppose $\operatorname EX_0=0$, $\operatorname E|X_0|^4<\infty$
and \autoref{a:eta} holds. If
\begin{equation} \label{e:m}
\frac{1}{m} + \frac{m^{5} (\log m)^2}{n^4} \to 0, \quad {\rm as } \
n \to \infty,
\end{equation}
then
\[
    2 \sqrt{m} \lp \widehat H - \frac{1}{2} \rp \convd {\cal N}(0,1),
    \quad {\rm as } \ n \to \infty.
\]
\end{theorem}

The  proof of~\autoref{t:LWE}
uses a similar argument as the proof of Theorem 2
of \citet{robinson:1995}.
One can reduce the argument to showing that \eqref{e:res1}
in \autoref{thm:res1} holds and
\begin{equation} \label{e:wtQn}
    \widetilde Q_n := \frac{1}{\sqrt{m}} \sum_{l = 1}^m \nu_{l,m}
    \lp \frac{I_n(\omega_l)}{f(0)} -1 \rp \convd {\cal N}(0,1), \quad
    {\rm as } \
     n \to \infty,
\end{equation}
where
\begin{equation}\label{eq:nu}
	\nu_{l,m}
	=\log l-\frac{1}{m} \sum_{j=1}^m \log j.
\end{equation}
The specific form of the weights $\nu_{l,m}$ follows from the
local Whittle likelihood \eqref{e:bp43}. We see that the
proof of \autoref{t:TS} can be reduced to known arguments,
\autoref{thm:res1} and
following result proven in  Section \ref{s:p4}.

\begin{theorem}\label{t:Q-t}
Convergence \eqref{e:wtQn} holds under the assumptions of \autoref{t:LWE}.
\end{theorem}

\section{A small simulation study} \label{s:sim}
Finite sample normality of spectral density estimators has been 
examined in many simulations studies, we refer to 
\citet{das:2021} for a study that uses a symmetrized version of $Q_n$, 
among many other estimators.
We focus here on simulations related to the change point 
problem explained in Section \ref{s:cpa}.
We examine the
finite sample behavior of the test statistic $T^{(R)}$ defined in
\eqref{e:TS} and the test based on it.
We consider two non-linear processes
used in modeling returns  on  financial assets. These models
are not considered by \citet{baek:pipiras:2012} and are not covered
by their theory.

\begin{definition}
The series $\{r_t\}$ is said to be a GARCH$(1,1)$
process if it satisfies the equations
\begin{align} \label{e:g11}
    r_t = \sigma_t \epsilon_t,\ \ \
\sigma_t^2 = \alpha_0 + \alpha_1 r_{t-1}^2 + \beta_1 \sigma_{t-1}^2,
\end{align}
where $\alpha_0 > 0$, $\alpha_1, \beta_1 \geq 0$,
$\alpha_1 + \beta_1 < 1$ and the
$\epsilon_t$ are iid random variables with mean zero, unit variance and
finite fourth moment.
\end{definition}

The condition $\alpha_1+\beta_1<1$ implies that the GARCH$(1,1)$ process has a second-order stationary solution (which is also strictly stationary) with a representation of form~\eqref{eq:defX_t}. A necessary and sufficient condition for the existence of the fourth moments of the GARCH$(1,1)$ process is given by
\begin{equation}\label{eq:garchm}
	\operatorname E|\beta_1+\alpha_1\epsilon_0^2|^2<1
\end{equation}
(we refer to \citet{francq:zakoian:2019} for an extensive review of
GARCH processes). \citet{wu:min:2005} established
(see Proposition~3 therein) that condition~\eqref{eq:garchm}
also implies that the GARCH$(1,1)$ process satisfies the
geometric-moment contraction (GMC) condition, i.e., $
\operatorname E|r_n-r_n^{(n)}|^4\le C\rho^n$,
for some $C<\infty$ and $\rho\in(0,1)$, where $r_n^{(n)}$
is defined by~\eqref{eq:Xtm} for $n\ge1$. Hence,
\autoref{a:eta} is satisfied provided that condition
\eqref{eq:garchm} holds.

As the next example, we consider the simplest stochastic volatility model.

\begin{definition} \label{d:SV}
The stochastic volatility (SV) model is given by
\begin{align} \label{e:sv}
    r_t = \exp(x_t/2) \epsilon_t,\ \ \
    x_t = \alpha + \phi x_{t-1} + w_t,
\end{align}
where $\alpha \in \mathbb{R}$,
$|\phi|<1$, $w_t \stackrel{iid}{\sim}\mathcal N(0,\sigma_w^2)$,
$\epsilon_t \stackrel{iid}{\sim} \mathcal N(0,1)$ and $\{w_t\}$
and $\{\epsilon_t\}$ are mutually independent.
\end{definition}

\begin{proposition}\label{p:ex}
The $r_t$ specified  in \autoref{d:SV}
satisfy \autoref{a:eta}.
\end{proposition}

We computed the empirical density
functions of the normalized LWE based on $\{r_t\}$
for both models, which are displayed in Figures \ref{f:g11} and \ref{f:sv}
presented in the Supporting Information.
These figures show that the asymptotic normality
established in \autoref{t:LWE}  holds well in finite samples.
We also evaluated the empirical size  of the test based on
\autoref{t:TS}. The results reported in Tables \ref{tb:g11} and \ref{tb:sv}
in the Supporting Information show that these sizes are
comparable with those obtained for linear processes.

We have applied the test of Section \ref{s:cpa} to returns 
and absolute returns on the S\&P500 index over a period that 
includeds the COVID pandemic.  
For any \[m = n^p, \ \ \ p \in \{0.60. 0.65, 0.70, 0.75\} < 0.8, \]
 c.f. 
condition \eqref{e:m}, the test accept $H_0$ for returns and rejects 
$H_0$ for absolute returns. It appears that the absolute returns
(that model volatility), are better explained by a long memory 
model. This agrees with many previous studies going back 
to the foundational  work of \citet{ding:1993}; \citet{dalla:2015}
provides many related references.

\section{Proofs of the results of Sections \ref{s:est}
and \ref{s:cpa} } \label{s:p}

We begin with a few  auxiliary lemmas that isolate  technical arguments
from our main proofs.

\subsection{Auxiliary lemmas} \label{ss:aux}
\begin{lemma}\label{lemma:prodofcos}
For $x,y\in\mathbb R$, $k,r\in\mathbb Z$, and $n\ge1$,
\begin{multline*}
	\sum_{t=1}^n\cos((t-x)\omega_k)\cos((t-y)\omega_r)\\
	=
	\begin{cases}
	\frac n2[\cos(x\omega_k-y\omega_r)+\cos(x\omega_k+y\omega_r)]&\text{if}\ k+r\in n\mathbb Z\ \text{and}\ k-r\in n\mathbb Z;\\
	\frac n2\cos(x\omega_k-y\omega_r)&\text{if}\ k+r\in\mathbb Z\setminus n\mathbb Z\ \text{and}\ k-r\in n\mathbb Z;\\
	\frac n2\cos(x\omega_k+y\omega_r)&\text{if}\ k+r\in n\mathbb Z\ \text{and}\ k-r\in\mathbb Z\setminus n\mathbb Z;\\
	0&\text{if}\ k+r\in\mathbb Z\setminus n\mathbb Z\ \text{and}\ k-r\in \mathbb Z\setminus n\mathbb Z.
	\end{cases}
\end{multline*}
\end{lemma}
\begin{proof}
Using the fact that $\cos x=(e^{ix}+e^{-ix})/2$ for $x\in\mathbb R$,
\begin{align*}
	&4\sum_{t=1}^n\cos((t-x)\omega_k)\cos((t-y)\omega_r)\\
	 &=\sum_{t=1}^n(e^{i(t-x)\omega_k}+e^{-i(t-x)\omega_k})(e^{i(t-y)\omega_r}+e^{-i(t-y)\omega_r})\\
	&=e^{-i(x\omega_k+y\omega_r)}\sum_{t=1}^n(e^{2\pi i/n})^{(k+r)t}
	+e^{-i(x\omega_k-y\omega_r)}\sum_{t=1}^n(e^{2\pi i/n})^{(k-r)t}\\
	&\ \ \ +e^{i(x\omega_k-y\omega_r)}\sum_{t=1}^n(e^{2\pi i/n})^{-(k-r)t}
	+e^{i(x\omega_k+y\omega_r)}\sum_{t=1}^n(e^{2\pi i/n})^{-(k+r)t}.
\end{align*}
We have that
\begin{multline*}
	e^{-i(x\omega_k+y\omega_r)}\sum_{t=1}^n(e^{2\pi i/n})^{(k+r)t}
	+e^{i(x\omega_k+y\omega_r)}\sum_{t=1}^n(e^{2\pi i/n})^{-(k+r)t}
	\\
	=
	\begin{cases}
	2n\cos(x\omega_k+y\omega_r)&\text{if}\ k+r\in n\mathbb Z;\\
	0&\text{if}\ k+r\notin n\mathbb Z.
	\end{cases}
\end{multline*}
Similarly,
\begin{multline*}
	e^{-i(x\omega_k-y\omega_r)}\sum_{t=1}^n(e^{2\pi i/n})^{(k-r)t}
	+e^{i(x\omega_k-y\omega_r)}\sum_{t=1}^n(e^{2\pi i/n})^{-(k-r)t}
	\\
	=
	\begin{cases}
	2n\cos(x\omega_k-y\omega_r)&\text{if}\ k-r\in n\mathbb Z;\\
	0&\text{if}\ k-r\notin n\mathbb Z.
	\end{cases}
\end{multline*}
The proof is complete.
\end{proof}

\begin{lemma}\label{lemma:prodofsin}
For $x,y\in\mathbb R$, $k,r\in\mathbb Z$, and $n\ge1$,
\begin{multline*}
	\sum_{t=1}^n\sin((t-x)\omega_k)\sin((t-y)\omega_r)\\
	=
	\begin{cases}
	\frac n2[\cos(x\omega_k-y\omega_r)-\cos(x\omega_k+y\omega_r)]&\text{if}\ k+r\in n\mathbb Z\ \text{and}\ k-r\in n\mathbb Z;\\
	\frac n2\cos(x\omega_k-y\omega_r)&\text{if}\ k+r\in\mathbb Z\setminus n\mathbb Z\ \text{and}\ k-r\in n\mathbb Z;\\
	-\frac n2\cos(x\omega_k+y\omega_r)&\text{if}\ k+r\in n\mathbb Z\ \text{and}\ k-r\in\mathbb Z\setminus n\mathbb Z;\\
	0&\text{if}\ k+r\in\mathbb Z\setminus n\mathbb Z\ \text{and}\ k-r\in \mathbb Z\setminus n\mathbb Z.
	\end{cases}
\end{multline*}
\end{lemma}
The proof of~\autoref{lemma:prodofsin} is analogous to the proof of \autoref{lemma:prodofcos} and thus is omitted.

\begin{lemma} \label{l:2}
    Let $\nu_{l,m} = \log l - \frac{1}{m} \sum_{j=1}^m \log j$. Then,
\begin{equation} \label{e:nu1}
    \lim_{m \to \infty}\frac{1}{m} \sum_{k=1}^m \nu_{k,m}^2 =1,
\end{equation}
\begin{equation} \label{e:nu2}
    \max_{1 \leq k \leq m} |\nu_{k,m}| = O(\log m).
\end{equation}
\end{lemma}

\begin{proof}
Note that
\[
\frac{1}{m}\sum_{k=1}^m \nu_{k,m}^2 \sim \frac{1}{m} \sum_{k=1}^m \lp \log \lp \frac{k}{m}\rp +1 \rp^2 \to \int_0^1 (\log x + 1)^2 dx = 1.
\]
Moreover, \eqref{e:nu2} follows directly from Lemma 2
in \citet{robinson:1995}. The proof is complete.
\end{proof}

\subsection{Central limit theorem for weighted sums of the periodogram ordinates}
Suppose that $\kappa_m=(\kappa_{1,m},\ldots,\kappa_{m,m})\in\mathbb R^m$ with $m\ge1$ and consider a weighted sum of the periodogram ordinates
\begin{equation}\label{eq:w_sum}
	W_n
	=\frac1{\|\kappa_m\|_2^2}\sum_{j=1}^m\kappa_{j,m}I_n(\omega_j),
\end{equation}
where $\|\kappa_m\|_p=(|\kappa_{1,m}|^p+\ldots+|\kappa_{m,m}|^p)^{1/p}$ for $p\ge1$ and the periodogram ordinates are defined in~\eqref{eq:dftandper}. Observe that $Q_n$ given in~\eqref{eq:Qn_w} and $\tilde Q_n$ given in~\eqref{e:wtQn} can be expressed as special cases of \eqref{eq:w_sum}. The next theorem establishes a central limit theorem for the appropriately standardized weighted sum of the periodogram ordinates.
\begin{theorem}\label{thm:generalCLT}
Assume that $\operatorname EX_0=0$, $\operatorname E|X_0|^4<\infty$, $\Theta_4<\infty$ with $\Theta_4$ defined by~\eqref{eq:Thetamp}, and \autoref{a:sublinear} holds.
Suppose that $\kappa_m=(\kappa_{1,m},\ldots,\kappa_{m,m})\in\mathbb R^m$ with $m\ge1$ such that
\begin{equation}\label{eq:assumption_kappa}
	\frac{\|\kappa_m\|_4}{\|\kappa_m\|_2}\to0
	\quad\text{and}\quad
	\frac{\sum_{j=1}^m\kappa_{j,m}^2\sin^2(\omega_j/2)}{\|\kappa_m\|_2^2}\to0
	\quad\text{as}\quad m\to\infty,
\end{equation}
where $\omega_j=2\pi j/n$ are the Fourier frequencies with $1\le j\le m$. Then
\[
	\frac{\|\kappa_m\|_2}{f(0)}(W_n-\operatorname EW_n)
	\xrightarrow d\mathcal N(0,1)
\]
as $n\to\infty$, where $W_n$ and $f(0)$ are defined in~\eqref{eq:w_sum} and~\eqref{eq:spectraldensity} respectively.
\end{theorem}
The proof of \autoref{thm:generalCLT} is based on a general central limit theorem for quadratic forms established by \citet{liu:wu:2010} (Thereom~6 therein) which we state here for the sake of convenience. $\mathbf 1_A:\mathbb R\to\{0,1\}$ denotes the indicator function of a set $A\subset\mathbb R$.

\begin{theorem}[Theorem~6 of \citet{liu:wu:2010}]\label{thm:liu:wu:2010:6}
Let $a_{n,j}=b_{n,j}e^{ij\lambda}$, where $\lambda\in\mathbb R$, $b_{n,j}\in\mathbb R$ with $b_{n,j}=b_{n,-j}$, and
\[
	T_n
	=\sum_{1\le j,j'\le n}a_{n,j-j'}X_jX_{j'}
	\quad\text{and}\quad
	\sigma_n^2
	=(1+\mathbf 1_{\mathbb Z}(\lambda/\pi))\sum_{k=1}^n\sum_{t=1}^nb_{n,t-k}^2,
\]
Assume that $\operatorname EX_0=0$, $\operatorname E|X_0|^4<\infty$, $\Theta_4<\infty$ with $\Theta_4$ defined by~\eqref{eq:Thetamp}, and
\begin{align}
	\max_{0\le t\le n}b_{n,t}^2
	&=o(\varsigma_n^2),
	\quad\text{where}\
	\varsigma_n^2
	=\sum_{t=1}^nb_{n,t}^2;\label{eq:liu:wu:2010:5.2}\\
	n\varsigma_n^2
	&=O(\sigma_n^2);\label{eq:liu:wu:2010:5.3}\\
	\sum_{k=1}^n\sum_{t=1}^{k-1}\Bigl|\sum_{j=1+k}^na_{n,k-j}a_{n,t-j}\Bigr|^2
	&=o(\varsigma_n^4);\label{eq:liu:wu:2010:5.4}\\
	\sum_{k=1}^n|b_{n,k}-b_{n,k-1}|^2
	&=o(\varsigma_n^2).\label{eq:liu:wu:2010:5.5}
\end{align}
Then for $0\le\lambda<2\pi$, $\sigma_n^{-1}(T_n-\operatorname ET_n)\xrightarrow dN(0,4\pi^2f^2(\lambda))$.
\end{theorem}
We are now ready to prove \autoref{thm:generalCLT}.
\begin{proof}[Proof of \autoref{thm:generalCLT}]
We use \autoref{thm:liu:wu:2010:6} with $\lambda=0$ to prove \autoref{thm:generalCLT}. $W_n$ is a quadratic form given by
\[
	W_n
=\frac1{\|\kappa_m\|_2^2}\sum_{j=1}^m\kappa_{j,m}\cdot\frac1{2\pi n}\sum_{s,t=1}^nX_sX_t\cos((s-t)\omega_j)
	=\sum_{s,t=1}^na_{n,s-t}X_sX_t,
\]
where
\[
	a_{n,t}
	=\frac1{2\pi n\|\kappa_m\|_2^2}\sum_{j=1}^m\kappa_{j,m}\cos(t\omega_j)
\]
for $n\ge1$ and $|t|<n$.

For sufficiently large values of $n$ and $k\in\mathbb Z$,
\begin{equation}\label{eq:sumofa^2}
	\sum_{t=1}^na_{n,t-k}^2
=\frac1{4\pi^2n^2\|\kappa_m\|_2^4}\sum_{t=1}^n\Bigl|\sum_{j=1}^m\kappa_{j,m}\cos((t-k)\omega_j)\Bigr|^2
	=\frac1{8\pi^2n\|\kappa_m\|_2^2}
\end{equation}
since
\begin{align*}
	\sum_{t=1}^n\Bigl|\sum_{j=1}^m\kappa_{j,m}\cos((t-k)\omega_j)\Bigr|^2
&=\sum_{j=1}^m\sum_{l=1}^m\kappa_{j,m}\kappa_{l,m}\sum_{t=1}^n\cos((t-k)\omega_j)\cos((t-k)\omega_l)\\
	&=\frac n2\|\kappa_m\|_2^2
\end{align*}
using \autoref{lemma:prodofcos} as $j+l\in \mathbb Z\setminus n\mathbb Z$ for sufficiently large values of $n$ due to \autoref{a:sublinear}.

Using~\eqref{eq:sumofa^2},
\[
	\sigma_n^2
	=2\sum_{k=1}^n\sum_{t=1}^na_{n,t-k}^2
	=\frac1{4\pi^2\|\kappa_m\|_2^2}
	\quad\text{and}\quad
	\varsigma_n^2
	=\sum_{t=1}^na_{n,t}^2
	=\frac1{8\pi^2n\|\kappa_m\|_2^2}
\]
for sufficiently large values of $n$.

Using the Cauchy-Schwarz inequality and the fact that $|\cos x|\le 1$ for $x\in\mathbb R$,
\begin{align*}
	\max_{0\le t\le n}a_{n,t}^2
	&=\frac1{4\pi^2n^2\|\kappa_m\|_2^4}\max_{0\le t\le n}\Bigl|\sum_{j=1}^m\kappa_{j,m}\cos(t\omega_j)\Bigr|^2\\
	&\le\frac1{4\pi^2n^2\|\kappa_m\|_2^4}\|\kappa_m\|_2^2\max_{0\le t\le n}\Bigl\{\sum_{j=1}^m\cos^2(t\omega_j)\Bigr\}\\
	&\le\frac m{4\pi^2n^2\|\kappa_m\|_2^2}
\end{align*}
and hence
\[
	\frac{\max_{0\le t\le n}a_{n,t}^2}{\varsigma_n^2}
	=\frac m{4\pi^2n^2\|\kappa_m\|_2^2}\cdot8\pi^2n\|\kappa_m\|_2^2
	=\frac{2m}n
	=o(1)
	\quad\text{as}\quad n\to\infty
\]
due to \autoref{a:sublinear} so that \eqref{eq:liu:wu:2010:5.2} is satisfied. \eqref{eq:liu:wu:2010:5.3} is also satisfied since
\[
	\frac{n\varsigma_n^2}{\sigma_n^2}
	=\frac{4\pi^2\|\kappa_m\|_2^2}{8\pi^2\|\kappa_m\|_2^2}
	=\frac12
\]
for sufficiently large values of $n$.

Next, we verify that \eqref{eq:liu:wu:2010:5.4} holds. We have that
\begin{align*}
	&\sum_{k=1}^n\sum_{t=1}^{k-1}\Bigl|\sum_{j=k+1}^na_{n,k-j}a_{n,t-j}\Bigr|^2=\\
&\le\sum_{k=1}^n\sum_{t=1}^n\Bigl[\sum_{j=1}^na_{n,k-j}^2a_{n,t-j}^2+2\sum_{j=k+1}^{n-1}\sum_{l=j+1}^na_{n,k-j}a_{n,t-j}a_{n,k-l}a_{n,t-l}\Bigr]\\
	&=\Bigl[\sum_{j=1}^n\sum_{k=1}^na_{n,k-j}^2\sum_{t=1}^na_{n,t-j}^2+2\sum_{k=1}^n\sum_{j=k+1}^{n-1}a_{n,k-j}\sum_{l=j+1}^na_{n,k-l}\Bigl(\sum_{t=1}^na_{n,t-j}a_{n,t-l}\Bigr)\Bigr]\\
	&=\Bigl[\sum_{j=1}^n\Bigl|\sum_{t=1}^na_{n,t-j}^2\Bigr|^2
	+2\sum_{k=1}^n\sum_{j=k+1}^{n-1}a_{n,k-j}\beta_{n,k,j}\Bigr],
\end{align*}
where
\[
	\beta_{n,k,j}
	\coloneq\sum_{l=j+1}^na_{n,k-l}\alpha_{n,j,l}
	\quad\text{and}\quad
	\alpha_{n,j,l}
	\coloneq\sum_{t=1}^na_{n,t-j}a_{n,t-l}.
\]
By the Cauchy-Schwarz inequality,
\begin{align*}
	\Bigl|2\sum_{k=1}^n\sum_{j=k+1}^{n-1}a_{n,k-j}\beta_{n,k,j}\Bigr|
	&\le2\sum_{k=1}^n\Bigl|\sum_{j=k+1}^{n-1}a_{n,k-j}\beta_{n,k,j}\Bigr|\\
	&\le2\sum_{k=1}^n\Bigl(\sum_{j=k+1}^{n-1}a_{n,k-j}^2\Bigr)^{1/2}\Bigl(\sum_{j=k+1}^{n-1}\beta_{n,k,j}^2\Bigr)^{1/2}\\
	&\le2\sum_{k=1}^n\Bigl(\sum_{j=1}^na_{n,k-j}^2\Bigr)^{1/2}\Bigl(\sum_{j=1}^n\beta_{n,k,j}^2\Bigr)^{1/2}
\end{align*}
and
\begin{align}
	\beta_{n,k,j}^2
	&=\Bigl|\sum_{l=j+1}^na_{n,k-l}\alpha_{n,j,l}\Bigr|^2\notag\\
	&\le\Bigl(\sum_{l=j+1}^na_{n,k-l}^2\Bigr)\Bigl(\sum_{l=j+1}^n\alpha_{n,j,l}^2\Bigr)\notag\\
	&\le\Bigl(\sum_{l=1}^na_{n,k-l}^2\Bigr)\Bigl(\sum_{l=1}^n\alpha_{n,j,l}^2\Bigr).\label{eq:betaCS}
\end{align}
Using \autoref{lemma:prodofcos} twice,
\begin{align}
	\sum_{l=1}^n\alpha_{n,j,l}^2
	&=\sum_{l=1}^n\Bigl|\sum_{t=1}^n
	\frac1{2\pi n\|\kappa_m\|_2^2}\sum_{r=1}^m\kappa_{r,m}\cos((t-j)\omega_r)
	\frac1{2\pi n\|\kappa_m\|_2^2}\sum_{s=1}^m\kappa_{s,m}\cos((t-l)\omega_s)\Bigr|^2\notag\\
	&=\frac1{16\pi^4n^4\|\kappa_m\|_2^8}\sum_{l=1}^n\Bigl|\sum_{r=1}^m\sum_{s=1}^m\kappa_{r,m}\kappa_{s,m}\sum_{t=1}^n\cos((t-j)\omega_r)\cos((t-l)\omega_s)\Bigr|^2\notag\\
&=\frac1{64\pi^4n^2\|\kappa_m\|_2^8}\sum_{l=1}^n\Bigl|\sum_{r=1}^m\kappa_{r,m}^2\cos((j-l)\omega_r)\Bigr|^2\notag\\
	&=\frac1{64\pi^4n^2\|\kappa_m\|_2^8}\sum_{l=1}^n\sum_{r=1}^m\sum_{s=1}^m\kappa_{r,m}^2\cos((j-l)\omega_r)\kappa_{s,m}^2\cos((j-l)\omega_s)\notag\\
	&=\frac1{64\pi^4n^2\|\kappa_m\|_2^8}\sum_{r=1}^m\sum_{s=1}^m\kappa_{r,m}^2\kappa_{s,m}^2\sum_{l=1}^n\cos((j-l)\omega_r)\cos((j-l)\omega_s)\notag\\
	&=\frac1{64\pi^4n^2\|\kappa_m\|_2^8}\|\kappa_m\|_4^4\cdot\frac n2\notag\\
	&=\frac{\|\kappa_m\|_4^4}{128\pi^4n\|\kappa_m\|_2^8}\label{eq:alphabound}
\end{align}
for sufficiently large values of $n$.

Using~\eqref{eq:betaCS}, \eqref{eq:sumofa^2}, and~\eqref{eq:alphabound},
\[
	\beta_{n,k,j}^2
	\le\Bigl(\sum_{l=1}^na_{n,k-l}^2\Bigr)\Bigl(\sum_{l=1}^n\alpha_{n,j,l}^2\Bigr)
	\le\frac1{8\pi^2n\|\kappa_m\|_2^2}\cdot\frac{\|\kappa_m\|_4^4}{128\pi^4n\|\kappa_m\|_2^8}
	=\frac{\|\kappa_m\|_4^4}{2^{10}\pi^6n^2\|\kappa_m\|_2^{10}}.
\]
Using~\eqref{eq:sumofa^2} again,
\[
	\sum_{j=1}^n\Bigl|\sum_{t=1}^na_{n,t-j}^2\Bigr|^2
	=\sum_{j=1}^n\Bigl|\frac1{8\pi^2n\|\kappa_m\|_2^2}\Bigr|^2
	=\frac1{64\pi^4n\|\kappa_m\|_2^4}
\]
and
\begin{align*}
	\Bigl|2\sum_{k=1}^n\sum_{j=k+1}^{n-1}a_{n,k-j}\beta_{n,k,j}\Bigr|
	&\le2\sum_{k=1}^n\Bigl(\sum_{j=1}^na_{n,k-j}^2\Bigr)^{1/2}\Bigl(\sum_{j=1}^n\beta_{n,k,j}^2\Bigr)^{1/2}\\
	&\le2\sum_{k=1}^n\Bigl(\frac1{8\pi^2n\|\kappa_m\|_2^2}\Bigr)^{1/2}
	\cdot\frac{n^{1/2}\|\kappa_m\|_4^2}{2^{5}\pi^3n\|\kappa_m\|_2^5}\\
	&=2n\frac1{2^{3/2}\pi n^{1/2}\|\kappa_m\|_2}
	\cdot\frac{n^{1/2}\|\kappa_m\|_4^2}{2^5\pi^3n\|\kappa_m\|_2^5}\\
	&=\frac{\|\kappa_m\|_4^2}{2^{11/2}\pi^4\|\kappa_m\|_2^6}
\end{align*}
so that
\begin{align*}
	&\frac{\sum_{k=1}^n\sum_{t=1}^{k-1}|\sum_{j=k+1}^na_{n,k-j}a_{n,t-j}|^2}{\sigma_n^4}=\\
	&=\sum_{k=1}^n\sum_{t=1}^{k-1}\Bigl|\sum_{j=k+1}^na_{n,k-j}a_{n,t-j}\Bigr|^2
	\cdot16\pi^4\|\kappa_m\|_2^4\\
	&\le\Bigl[\frac1{64\pi^4n\|\kappa_m\|_2^4}
	+\frac{\|\kappa_m\|_4^2}{2^{11/2}\pi^4\|\kappa_m\|_2^6}\Bigr]
	\cdot16\pi^4\|\kappa_m\|_2^4\\
	&=\frac1{4n}
	+2^{-3/2}\biggl(\frac{\|\kappa_m\|_4}{\|\kappa_m\|_2}\biggr)^2\\
	&=o(1)
\end{align*}
as $n\to\infty$ due to \eqref{eq:assumption_kappa} and hence \eqref{eq:liu:wu:2010:5.4} holds.

We finally verify that \eqref{eq:liu:wu:2010:5.5} holds. Using the identity
\[
	\cos\theta-\cos(\varphi)
	=-2\sin\Bigl(\frac{\theta+\varphi}2\Bigr)\sin\Bigl(\frac{\theta-\varphi}2\Bigr)
\]
for $\theta,\varphi\in\mathbb R$ and \autoref{lemma:prodofsin},
\begin{align*}
	&\sum_{k=1}^n|a_{n,k}-a_{n,k-1}|^2=\\
&=\frac1{4\pi^2n^2\|\kappa_m\|_2^4}\sum_{k=1}^n\Bigl|\sum_{j=1}^m\kappa_{j,m}[\cos(k\omega_j)-\cos((k-1)\omega_j)]\Bigr|^2\\
	&=\frac1{\pi^2n^2\|\kappa_m\|_2^4}\sum_{k=1}^n\Bigl|\sum_{j=1}^m\kappa_{j,m}[\sin((2k-1)\omega_j/2)\sin(\omega_j/2)]\Bigr|^2\\
	&=\frac1{\pi^2n^2\|\kappa_m\|_2^4}\sum_{j,l=1}^m\kappa_{j,m}\kappa_{l,m}\sin(\omega_j/2)\sin(\omega_l/2)\Bigl\{\sum_{k=1}^n\sin((k-1/2)\omega_j)\sin((k-1/2)\omega_l)\Bigr\}\\
	&=\frac1{2\pi^2n\|\kappa_m\|_2^4}\sum_{j=1}^m\kappa_{j,m}^2\sin^2(\omega_j/2)
\end{align*}
for sufficiently large values of $n$. Hence,
\begin{align*}
	\frac{\sum_{k=1}^n|a_{n,k}-a_{n,k-1}|^2}{\varsigma_n^2}
	&=\frac1{2\pi^2n\|\kappa_m\|_2^4}\sum_{j=1}^m\kappa_{j,m}^2\sin^2(\omega_j/2)
	\cdot8\pi^2n\|\kappa_m\|_2^2\\
	&=\frac{4\sum_{j=1}^m\kappa_{j,m}^2\sin^2(\omega_j/2)}{\|\kappa_m\|_2^2}\\
	&=o(1)
\end{align*}
as $n\to\infty$ due to \eqref{eq:assumption_kappa} and hence \eqref{eq:liu:wu:2010:5.5} holds. \autoref{thm:liu:wu:2010:6} then implies that
\[
	\sigma_n^{-1}(W_n-\operatorname EW_n)
	\xrightarrow dN(0,4\pi^2f^2(0))
	\quad\text{as}\quad
	n\to\infty,
\]
which completes the proof.
\end{proof}

\subsection{Proofs of the theorems of Section \ref{s:est}} \label{ss:p3}
\begin{proof} [\sc Proof of \autoref{t:clt1}:]
We set $\kappa_{j,m}=1$ in \autoref{thm:generalCLT} for $j=1,\ldots,m$ and $m\ge1$. The first condition of~\eqref{eq:assumption_kappa} holds. Using the inequality $|\sin x|\le|x|$ for $x\in\mathbb R$,

\[
	\frac{\sum_{j=1}^m\sin^2(\omega_j/2)}m
	\le\frac{\pi^2 \sum_{j=1}^m j^2}{n^2m}
	=O\Bigl(\Bigl(\frac mn\Bigr)^2\Bigr)
\]
as $n\to\infty$ and the second condition of~\eqref{eq:assumption_kappa} also holds under \autoref{a:sublinear}. The proof is complete.
\end{proof}

\begin{proof}[\sc Proof of~\autoref{thm:convrate}]
Observe that $Q_n$ can be expressed in the following way
\begin{equation}\label{eq:Q_nlrv}
	Q_n
	=\frac1{2\pi}\sum_{h=-(n-1)}^{n-1}a_{n,h}\hat\gamma(h)
\end{equation}
for $n\ge1$, where
\[
	a_{n,h}
	= \frac1m\sum_{j=1}^m\cos(h\omega_j)
\]
for $n\ge1$ and $|h|<n$ while $\hat\gamma(h)$ are sample autocovariances given by
\[
	\hat\gamma(h)
	= \frac1n\sum_{j=1}^{n-h}X_{j+h}X_j
\]
for $0\le h<n$ and $\hat\gamma(h)=\hat\gamma(-h)$ for $-n<h\le 0$. Hence, we have that
\begin{align*}
	&\operatorname E\biggl[\frac1m\sum_{j=1}^m I_n(\omega_j)\biggr]-f(0)\\
	 &=\frac1{2\pi}\sum_{h=-(n-1)}^{n-1}\Bigl(1-\frac{|h|}n\Bigr)\gamma(h)\biggl[\frac1m\sum_{j=1}^m\cos(h\omega_j)\biggr]
	-\frac1{2\pi}\sum_{h=-\infty}^\infty\gamma(h)\\
	 &=\frac1{2\pi}\biggl[\,\sum_{h=-(n-1)}^{n-1}\biggl\{\Bigl(1-\frac{|h|}n\Bigr)\biggl[\frac1m\sum_{j=1}^m\cos(h\omega_j)\biggr]-1\biggr\}\gamma(h)
	-\sum_{|h|\ge n}\gamma(h)\biggr]\\
	 &=\frac1\pi\biggl[\,\sum_{h=1}^{n-1}\biggl\{\biggl[\frac1m\sum_{j=1}^m\cos(h\omega_j)\biggr]
	-\frac hn\biggl[\frac1m\sum_{j=1}^m\cos(h\omega_j)\biggr]
	-1\biggr\}\gamma(h)
	-\sum_{h=n}^\infty\gamma(h)\biggr]\\
	 &=\frac1\pi\biggl[\,\sum_{h=1}^{n-1}\biggl\{\biggl[\frac1m\sum_{j=1}^m\cos(h\omega_j)\biggr]
	-1
	-\frac hn\biggl[\frac1m\sum_{j=1}^m\cos(h\omega_j)\biggr]\biggr\}\gamma(h)
	-\sum_{h=n}^\infty\gamma(h)\biggr].
\end{align*}
Split the sum into two parts
\begin{multline}\label{eq:splitsum}
	 \sum_{h=1}^{n-1}\biggl|\frac1m\sum_{j=1}^m\cos(h\omega_j)-1\biggr||\gamma(h)|\\
	 =\sum_{h=1}^{\tau_{n,m}-1}\biggl|\frac1m\sum_{j=1}^m\cos(h\omega_j)-1\biggr||\gamma(h)|
	 +\sum_{h=\tau_{n,m}}^{n-1}\biggl|\frac1m\sum_{j=1}^m\cos(h\omega_j)-1\biggr||\gamma(h)|,
\end{multline}
where
\[
	\tau_{n,m}
	= \biggl\lfloor\frac{\sqrt 3n}{\pi\sqrt{(m+1)(2m+1)}}\biggr\rfloor.
\]
Using the mean value theorem and the fact that $|\sin x|\le|x|$ for $x\in\mathbb R$,
\[
	\Bigl|\cos\Bigl(h\cdot\frac{2\pi j}n\Bigr)-\cos(0)\Bigr|
	=|\sin c|\Bigl|\frac{2\pi hj}n\Bigr|
	\le\Bigl|\frac{2\pi hj}n\Bigr|^2
\]
with some $c\in(0,2\pi hj/n)$. Hence,
\begin{align}
	 \sum_{h=1}^{\tau_{n,m}-1}\biggl|\frac1m\sum_{j=1}^m\cos(h\omega_j)-1\biggr||\gamma(h)|
	&\le\sum_{h=1}^{\tau_{n,m}-1}\frac23\Bigl|\frac{\pi h}n\Bigr|^2(m+1)(2m+1)|\gamma(h)|\notag\\
	&=\frac23\Bigl|\frac\pi n\Bigr|^2(m+1)(2m+1)\sum_{h=1}^{\tau_{n,m}-1}h^2|\gamma(h)|.\label{eq:smallterms}
\end{align}
If $0\le r<2$,
\[
	\sum_{h=1}^{\tau_{n,m}-1}h^2|\gamma(h)|
	=\sum_{h=1}^{\tau_{n,m}-1}h^{2-r}h^r|\gamma(h)|
	\le(\tau_{n,m}-1)^{2-r}\sum_{h=1}^{\tau_{n,m}-1}h^r|\gamma(h)|
\]
and~\eqref{eq:smallterms} is $O((m/n)^r)$ as $n\to\infty$. If $r\ge2$, \eqref{eq:smallterms} is $O((m/n)^2)$ as $n\to\infty$. We bound the second term of~\eqref{eq:splitsum} in the following way
\begin{align*}
	 \sum_{h=\tau_{n,m}}^{n-1}\biggl|\frac1m\sum_{j=1}^m\cos(h\omega_j)-1\biggr||\gamma(h)|
	&\le2\sum_{h=\tau_{n,m}}^{n-1}|\gamma(h)|\\
	&=2\sum_{h=\tau_{n,m}}^{n-1}h^{-r}h^r|\gamma(h)|\\
	&\le2\tau_{n,m}^{-r}\sum_{h=\tau_{n,m}}^{n-1}h^r|\gamma(h)|\\
	&=o((m/n)^r)
\end{align*}
as $n\to\infty$. We have that
\[
	\sum_{h=1}^{n-1}\frac hn\biggl|\frac1m\sum_{j=1}^m\cos(h\omega_j)\biggr||\gamma(h)|
	\le\frac1n\sum_{h=1}^{n-1}h|\gamma(h)|.
\]
If $0<r<1$, then
\[
	\frac1n\sum_{h=1}^{n-1}h|\gamma(h)|
	=\frac1n\sum_{h=1}^{n-1}h^{1-r}h^r|\gamma(h)|
	\le\frac{(n-1)^{1-r}}n\sum_{h=1}^{n-1}h^r|\gamma(h)|
	=O(n^{-r})
\]
as $n\to\infty$. If $r\ge1$, then
\[
	\frac1n\sum_{h=1}^{n-1}h|\gamma(h)|
	=O(n^{-1})
\]
as $n\to\infty$.
Finally,
\[
	\sum_{h=n}^\infty|\gamma(h)|
	=\sum_{h=n}^\infty h^rh^{-r}|\gamma(h)|
	\le n^{-r}\sum_{h=n}^\infty h^r|\gamma(h)|
	=o(n^{-r})
\]
as $n\to\infty$. The proof is complete.
\end{proof}

\subsection{Proofs of the results of Section \ref{s:cpa}} \label{s:p4}

\begin{proof}[\sc Proof of~\autoref{t:Q-t}]
We set $\kappa_{j,m}=\nu_{j,m}$  in \autoref{thm:generalCLT} with $\nu_{j,m}$ given by~\eqref{eq:nu} for $j=1,\ldots,m$ and $m\ge1$. Using \autoref{l:2} and the inequality $|\sin x|\le|x|$ for $x\in\mathbb R$,
\[
	\frac{\|\nu_m\|_4}{\|\nu_m\|_2}
	\le\frac{m^{1/4}\max_{1\le j\le m}|\nu_{j,m}|}{m^{1/2}\cdot m^{-1/2}\|\nu_m\|_2}
	=O\Bigl(\frac{\log m}{m^{1/4}}\Bigr)
\]
and
\[
	\frac{\sum_{j=1}^m\nu_{j,m}^2\sin^2(\omega_j/2)}{\|\nu_m\|_2^2}
\le\frac{\pi\max_{1\le j\le m}\nu_{j,m}^2\cdot\sum_{j=1}^mj^2}{n^2\|\nu_m\|_2^2}
	=O\Bigl(\frac{m^2\log^2m}{n^2}\Bigr)
\]
as $n\to\infty$ so that conditions~\eqref{eq:assumption_kappa} are satisfied under assumption~\eqref{e:m}. \autoref{thm:generalCLT} then implies that
\[
	\frac{\|\nu_m\|_2}{f(0)}
	\Bigl[\frac1{\|\nu_n\|_2^2}\sum_{j=1}^m\nu_{j,m}I_n(\omega_j)-\operatorname E\Bigl[\frac1{\|\nu_n\|_2^2}\sum_{j=1}^m\nu_{j,m}I_n(\omega_j)\Bigr]\Bigr]
	\xrightarrow d
	\mathcal N(0,1)
	\quad\text{as}\quad
	n\to\infty.
\]
To complete the proof, we show that
\[
	E\Bigl[\frac1{\sqrt m}\sum_{j=1}^m\nu_{j,m}I_n(\omega_j)\Bigr]
	=o(1)
	\quad\text{as}\quad n\to\infty.
\]
We have that
\begin{align*}
	\operatorname E\Bigl[\frac1{\sqrt m}\sum_{l=1}^m\nu_{l,m}I_n(\omega_l)\Bigr]
	&=\frac1{\sqrt m}\sum_{l=1}^m\nu_{l,m}\operatorname EI_n(\omega_l)\\
	&=\frac1{2\pi}\sum_{h=-(n-1)}^{n-1}\Bigl[\frac1{\sqrt m}\sum_{l=1}^m\nu_{l,m}\cos(h\omega_l)\Bigr]\Bigl(1-\frac{|h|}n\Bigr)\gamma(h)\\
&=\frac1\pi\sum_{h=1}^{n-1}\Bigl[\frac1{\sqrt m}\sum_{l=1}^m\nu_{l,m}\{\cos(h\omega_l)-1\}\Bigr]\Bigl(1-\frac{h}n\Bigr)\gamma(h)
\end{align*}
since $\sum_{j=1}^m\nu_{j,m}=0$ for $m\ge1$. Using \autoref{l:2}, the identity $1-\cos x=2\sin^2(x/2)$ for $x\in\mathbb R$, and the inequality $|\sin x|\le |x|$ for $x\in\mathbb R$,
\begin{align*}
	\Bigl|\operatorname E\Bigl[\frac1{\sqrt m}\sum_{l=1}^m\nu_{l,m}I_n(\omega_l)\Bigr]\Bigr|
	&\le\frac1\pi\sum_{h=1}^{n-1}\frac1{\sqrt m}\sum_{l=1}^m|\nu_{l,m}||\cos(h\omega_l)-1||\gamma(h)|\\
	&\le\frac{2C\log m}\pi\sum_{h=1}^{n-1}\frac1{\sqrt m}\sum_{l=1}^m\Bigl(\frac{h\omega_l}{2}\Bigr)^2|\gamma(h)|\\
	&=\frac{2\pi C}3\sum_{h=1}^{n-1}h^2|\gamma(h)|\cdot\frac{\log m}{n^2\sqrt m}\sum_{l=1}^ml^2\\
	&=O\Bigl(\frac{m^{5/2}\log m}{n^2}\Bigr)
\end{align*}
as $n\to\infty$ since $\sum_{h=1}^\infty h^2|\gamma(h)|<\infty$ under \autoref{a:eta} (see inequality~\eqref{e:h3g}). The proof is complete.
\end{proof}

\begin{proof}[\sc Proof of~\autoref{p:ex}]
Let $\mu = E[x_t] = \frac{\alpha}{1-\phi}$, then
$x_t = \mu + \sum_{j=0}^{\infty} \phi^j w_{t-j}$.
Define
\[
    u_t = \exp \lp \frac{x_t}{2} \rp
    = \exp \lp \frac{1}{2} \mu + \frac{1}{2}
    \sum_{j=0}^{\infty} \phi^j w_{t-j} \rp\
    = \exp \lp \zeta_{m-1,t} + \xi_{m,t} \rp,
\]
where
\[
\zeta_{m-1,t} = \frac{1}{2}\mu + \frac{1}{2}
\sum_{j=0}^{m-1}\phi^j w_{t-j}, \ \ \
\xi_{m,t} = \frac{1}{2} \sum_{j=m}^{\infty} \phi^j w_{t-j}.
\]
Clearly, the mapping from $(w_t, w_{t-1},\ldots)$ to $u_t$ is measurable. For each $m \in \mathbb{N}^+$, we can approximate the $u_t$ by
\begin{equation} \label{e:ut-m}
u_t^{(m)} = \exp \lp \zeta_{m-1,t} + \xi_{m,t}^{\prime} \rp,
\end{equation}
where
\[
 \xi_{m,t}^{\prime} = \frac{1}{2} \sum_{j=m}^{\infty} \phi^j w_{t-j}^{\prime},
\]
and $\{ w_{t-l}^{\prime},\ l\geq m, m \in \mathbb{N}^+\}$
are iid copies of $w_0$.

Using the  identity $(a-b)^4 = a^4-4a^3b + 6a^2b^2-4ab^3+b^4$, we have
\begin{align*}
    \left|u_t - u_t^{(m)} \right|^4 &= \exp \lp 4\zeta_{m-1,t} \rp \cdot \left[ \exp \lp \xi_{m,t} \rp - \exp \lp \xi_{m,t}^{\prime} \rp \right]^4\\
    &=  \exp \lp 4\zeta_{m-1,t} \rp \cdot \big\{ \exp \lp 4\xi_{m,t} \rp -4 \exp \lp 3\xi_{m,t} + \xi_{m,t}^{\prime} \rp \\
    &\quad + 6 \exp  \lp 2\xi_{m,t}  + 2\xi_{m,t}^{\prime} \rp - 4 \exp \lp \xi_{m,t} + 3\xi_{m,t}^{\prime} \rp + \exp \lp 4\xi_{m,t}^{\prime} \rp \big\}.
\end{align*}
Let $\varphi(s)$ be the characteristic function of $w_0$.
Then, the characteristic function of the $\xi_{m,t}$ is given by
\[
    \varphi_{\xi}(s) = E[\exp \lp is \xi_{m,t} \rp]
    = \prod_{j=m}^{\infty} \varphi(s\phi^j)
    = \exp \left\{ -\frac{1}{2} \lp \frac{\phi^{2m}}{1-\phi^2}
    \sigma_w^2 \rp s^2 \right\}.
\]
Hence, 
\[
\xi_{m,t} \stackrel{d}{=} \xi_{m,t}^{\prime} \sim {\cal N} \lp 0, \frac{\phi^{2m}}{4(1-\phi^2)} \sigma_w^2 \rp.
\]
On the other hand,
\[
\zeta_{m-1,t} \sim {\cal N} \lp \frac{1}{2} \mu, \frac{1-\phi^{2m}}{4(1-\phi^2)} \sigma_w^2 \rp.
\]
We know that $E[e^{\mu + \sigma Z}] = \exp \lp \mu + \frac{1}{2} \sigma^2 \rp$ if $Z \sim {\cal N}(0,1)$. Consequently,
\begin{align} \label{e:ED-ut}
    E\left|u_t - u_t^{(m)} \right|^4 &= \exp \lp 2\mu + \frac{2(1-\phi^{2m})}{1-\phi^2} \sigma_w^2 \rp \cdot \Bigg\{ 2\exp \lp
\frac{2\phi^{2m}}{1-\phi^2}  \sigma_w^2 \rp \nonumber \\
&\quad - 8 \exp \lp
\frac{5\phi^{2m}}{2(1-\phi^2)}  \sigma_w^2 \rp + 6 \exp \lp
\frac{2\phi^{2m}}{1-\phi^2}  \sigma_w^2 \rp \Bigg\} \nonumber \\
&= \exp \lp 2\mu + \frac{2\sigma_w^2}{1-\phi^2} \rp \cdot \Bigg\{ 8  - 8 \exp \lp \frac{\phi^{2m}}{2(1-\phi^2)} \sigma_w^2 \rp \Bigg\} \nonumber \\
&= O(\phi^{2m}),
\end{align}
where the last equality comes from the Taylor expansion.

Finally, let $\boldsymbol{\eta}_t = (w_t, \epsilon_t)^\intercal$. Since $r_t = u_t \epsilon_t$, the mapping from $(\boldsymbol{\eta}_t, \boldsymbol{\eta}_{t-1}, \ldots)$ to $r_t$ is measurable. For each $m \in \mathbb{N}^+$, the $r_t$ can be approximated by $r_t^{(m)} = u_t^{(m)} \epsilon_t$, where the $u_t^{(m)}$ is given by \eqref{e:ut-m}. Since $\{\epsilon_t\}$ and $\{w_t\}$ are mutually independent,
\begin{align*}
    \| r_t - r_t^{(m)} \|_4 &= \|\epsilon_0 \|_4 \cdot \|u_t - u_t^{(m)}\|_4\\
    &= O\lp \phi^{m/2} \rp \\
    &= O(m^{-\kappa^*})
\end{align*}
for any $\kappa^* > 0$. The proof is complete.
\end{proof}

\bigskip\centerline{\sc Acknowledgement}

\noindent
P.\ Kokoszka and X.\ Meng  were  partially supported by the United
States National Science Foundation grants DMS--2123761 and 
DMS--2412408.
The authors  thank  two anonymous reviewers, as well as the
co-editor, for their comments and suggestions that helped 
improve this work substantially.

\bigskip\centerline{\sc  Data Availability Statement}
\noindent  This paper does not use any data.

\bigskip\centerline{\sc  Supporting Information}
\noindent  Additional Supporting Information may be found
online in the supporting information tab for this article.
It contains graphs and tables related to the test discussed in
Section~\ref{s:cpa}.

\bigskip
\small\renewcommand{\baselinestretch}{0.95}
\bibliography{bibV,neda}

\newpage

\normalsize\renewcommand{\baselinestretch}{1.0}

\setcounter{page}{1}

\centerline{\Large \it Supporting Information}

\appendix

\section*{Finite sample behavior  of the test based on Theorem~6}

Figures~\ref{f:g11} and~\ref{f:sv} show the empirical density
functions of the normalized LWE based on $\{r_t\}$
for both GARCH(1,1) and SV models.
These densities are asymptotically standard normal.

Tables~\ref{tb:g11} and~\ref{tb:sv} show
the empirical size of the test based on Theorem~6.
We set  $\mu = 0$
in~(19) and consider $r_t$ that follow  either the GARCH$(1,1)$
model or the  SV model with the same parameters as in Figures~\ref{f:g11} and~\ref{f:sv}.
The jumps sizes
$\Delta = 0.3$ and $0.6$  are close, respectively,  to 0.25 and 0.5
of the standard deviation of the $r_t$. The empirical sizes
Tables~\ref{tb:g11} and~\ref{tb:sv}
are  comparable with those obtained for linear processes.

\begin{figure}[!htbp]\centering
\includegraphics[scale = 0.9]{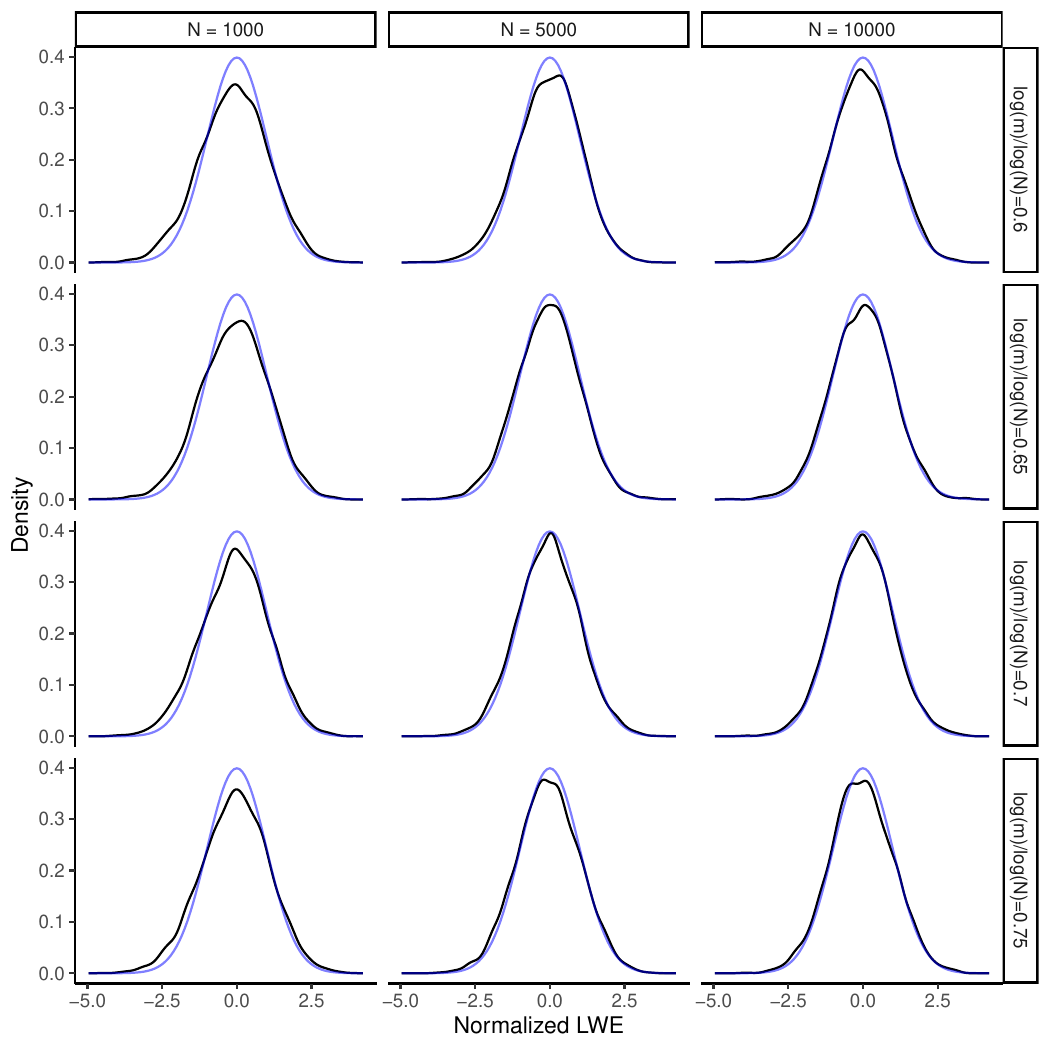}
\caption{ Empirical density functions of $T^{(R)}$
based on $5{,}000$ replications with the standard normal density
overlaid in blue.  The
$r_t$ follow  GARCH$(1,1)$ model
with $\alpha_0 = 0.5$, $\alpha_1 = 0.2$
and $\beta_1 = 0.4$.
\label{f:g11}}
\end{figure}

\begin{figure}[!htbp]\centering
\includegraphics[scale = .9]{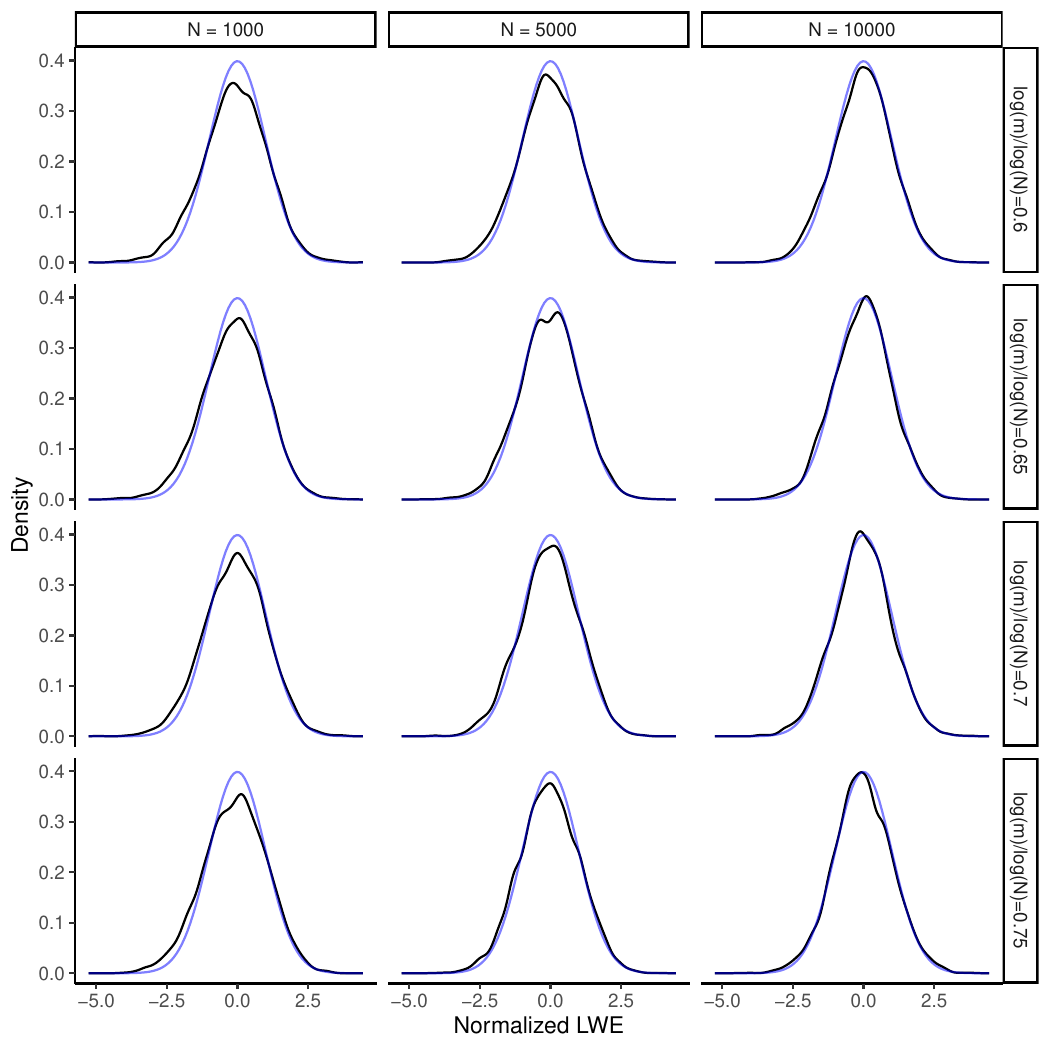}
\caption{The same as \autoref{f:g11}, but the $r_t$ follow
SV model  with $\alpha = 0$, $\phi = 0.5$
and $\sigma_w^2 = 1$.  \label{f:sv}}
\end{figure}

\clearpage

\begin{table}[!htbp]\centering
{\footnotesize
\caption{Empirical size (in percent) for observations from the GARCH$(1,1)$ model, based on $5{,}000$ replications.
 }  \label{tb:g11}
\begin{tabular}{ccrrrrrrrrrrr}
\hline
Nominal size   &                      &                          & 1.0                      &                           &                      &                          & 5.0                      &                           &                      &                          & 10.0                     &                           \\ \hline
$N$            &                      & \multicolumn{1}{c}{1000} & \multicolumn{1}{c}{5000} & \multicolumn{1}{c}{10000} & \multicolumn{1}{c}{} & \multicolumn{1}{c}{1000} & \multicolumn{1}{c}{5000} & \multicolumn{1}{c}{10000} & \multicolumn{1}{c}{} & \multicolumn{1}{c}{1000} & \multicolumn{1}{c}{5000} & \multicolumn{1}{c}{10000} \\ \hline
               &                      & \multicolumn{11}{c}{$\Delta = 0.30, \theta = 0.25$}                                                                                                                                                                                                                                               \\ \cline{3-13}
$m = N^{0.60}$ &                      & 0.90                     & 1.12                     & 1.68                      &                      & 4.00                     & 5.00                     & 6.48                      &                      & 7.86                     & 9.96                     & 11.56                     \\
$m = N^{0.65}$ &                      & 0.88                     & 1.14                     & 1.14                      &                      & 3.92                     & 4.84                     & 5.80                      &                      & 7.68                     & 9.22                     & 10.44                     \\
$m = N^{0.70}$ & \multicolumn{1}{l}{} & 0.92                     & 1.18                     & 1.46                      &                      & 3.98                     & 5.00                     & 5.58                      &                      & 7.86                     & 9.68                     & 9.90                      \\
$m = N^{0.75}$ &                      & 0.98                     & 1.14                     & 1.14                      &                      & 4.38                     & 5.30                     & 5.22                      &                      & 8.34                     & 9.98                     & 10.06                     \\ \hline
               &                      & \multicolumn{11}{c}{$\Delta = 0.30, \theta = 0.50$}                                                                                                                                                                                                                                               \\ \cline{3-13}
$m = N^{0.60}$ &                      & 0.64                     & 0.68                     & 0.78                      &                      & 3.14                     & 3.36                     & 4.20                      &                      & 6.26                     & 6.58                     & 8.28                      \\
$m = N^{0.65}$ &                      & 0.54                     & 0.72                     & 0.50                      &                      & 3.02                     & 3.26                     & 3.72                      &                      & 5.90                     & 6.88                     & 7.70                      \\
$m = N^{0.70}$ & \multicolumn{1}{l}{} & 0.80                     & 0.70                     & 0.98                      &                      & 3.38                     & 3.96                     & 4.06                      &                      & 6.68                     & 7.36                     & 7.66                      \\
$m = N^{0.75}$ &                      & 0.86                     & 0.64                     & 0.88                      &                      & 3.78                     & 3.72                     & 3.84                      &                      & 7.22                     & 7.58                     & 8.02                      \\ \hline
               &                      & \multicolumn{11}{c}{$\Delta = 0.30, \theta = 0.75$}                                                                                                                                                                                                                                               \\ \cline{3-13}
$m = N^{0.60}$ &                      & 0.76                     & 1.08                     & 1.72                      &                      & 3.66                     & 5.02                     & 5.96                      &                      & 7.26                     & 10.00                    & 11.08                     \\
$m = N^{0.65}$ &                      & 0.64                     & 1.20                     & 1.24                      &                      & 3.38                     & 4.68                     & 5.30                      &                      & 7.28                     & 9.20                     & 10.42                     \\
$m = N^{0.70}$ & \multicolumn{1}{l}{} & 0.80                     & 1.12                     & 1.62                      &                      & 3.88                     & 4.90                     & 5.66                      &                      & 7.80                     & 9.38                     & 10.18                     \\
$m = N^{0.75}$ &                      & 0.90                     & 0.98                     & 1.30                      &                      & 4.20                     & 4.86                     & 5.60                      &                      & 8.24                     & 9.88                     & 10.08                     \\ \hline
               &                      & \multicolumn{11}{c}{$\Delta = 0.60, \theta = 0.25$}                                                                                                                                                                                                                                               \\ \cline{3-13}
$m = N^{0.60}$ &                      & 1.68                     & 1.76                     & 1.58                      &                      & 6.24                     & 6.04                     & 6.00                      &                      & 11.12                    & 11.10                    & 11.32                     \\
$m = N^{0.65}$ &                      & 1.74                     & 1.50                     & 1.26                      &                      & 6.02                     & 6.08                     & 6.02                      &                      & 10.94                    & 10.44                    & 10.34                     \\
$m = N^{0.70}$ & \multicolumn{1}{l}{} & 1.52                     & 1.44                     & 1.64                      &                      & 5.70                     & 6.06                     & 6.04                      &                      & 10.78                    & 11.04                    & 10.24                     \\
$m = N^{0.75}$ &                      & 1.50                     & 1.36                     & 1.54                      &                      & 6.12                     & 5.88                     & 5.84                      &                      & 11.12                    & 11.06                    & 10.62                     \\ \hline
               &                      & \multicolumn{11}{c}{$\Delta = 0.60, \theta = 0.5$}                                                                                                                                                                                                                                                \\ \cline{3-13}
$m = N^{0.60}$ &                      & 0.78                     & 0.74                     & 0.72                      &                      & 3.82                     & 3.74                     & 4.32                      &                      & 7.68                     & 7.24                     & 8.62                      \\
$m = N^{0.65}$ &                      & 0.66                     & 0.78                     & 0.60                      &                      & 3.68                     & 3.56                     & 4.08                      &                      & 7.30                     & 7.48                     & 8.20                      \\
$m = N^{0.70}$ & \multicolumn{1}{l}{} & 0.86                     & 0.84                     & 1.00                      &                      & 3.98                     & 4.18                     & 4.20                      &                      & 7.60                     & 7.82                     & 8.12                      \\
$m = N^{0.75}$ &                      & 0.98                     & 0.70                     & 0.94                      &                      & 4.38                     & 4.30                     & 4.06                      &                      & 8.10                     & 8.18                     & 8.52                      \\ \hline
               &                      & \multicolumn{11}{c}{$\Delta = 0.60, \theta = 0.75$}                                                                                                                                                                                                                                               \\ \cline{3-13}
$m = N^{0.60}$ &                      & 1.58                     & 1.82                     & 1.66                      &                      & 6.16                     & 6.32                     & 6.26                      &                      & 10.84                    & 11.14                    & 11.46                     \\
$m = N^{0.65}$ &                      & 1.38                     & 1.52                     & 1.54                      &                      & 5.70                     & 5.64                     & 5.94                      &                      & 10.42                    & 10.52                    & 10.70                     \\
$m = N^{0.70}$ & \multicolumn{1}{l}{} & 1.32                     & 1.58                     & 1.94                      &                      & 5.86                     & 6.12                     & 5.98                      &                      & 10.86                    & 10.48                    & 10.66                     \\
$m = N^{0.75}$ &                      & 1.56                     & 1.50                     & 1.54                      &                      & 5.56                     & 5.94                     & 6.06                      &                      & 10.92                    & 11.02                    & 11.02                     \\ \hline
\end{tabular}}
\end{table}

\begin{table}[!htbp]\centering
{\footnotesize
\caption{Empirical size (in percent) for observations from the SV model, based on $5{,}000$ replications.
 }  \label{tb:sv}
\begin{tabular}{ccrrrrrrrrrrr}
\hline
Nominal size   &                      &                          & 1.0                      &                           &                      &                          & 5.0                      &                           &                      &                          & 10.0                     &                           \\ \hline
$N$            &                      & \multicolumn{1}{c}{1000} & \multicolumn{1}{c}{5000} & \multicolumn{1}{c}{10000} & \multicolumn{1}{c}{} & \multicolumn{1}{c}{1000} & \multicolumn{1}{c}{5000} & \multicolumn{1}{c}{10000} & \multicolumn{1}{c}{} & \multicolumn{1}{c}{1000} & \multicolumn{1}{c}{5000} & \multicolumn{1}{c}{10000} \\ \hline
               &                      & \multicolumn{11}{c}{$\Delta = 0.30, \theta = 0.25$}                                                                                                                                                                                                                                               \\ \cline{3-13}
$m = N^{0.60}$ &                      & 0.72                     & 1.08                     & 1.22                      &                      & 2.74                     & 4.94                     & 5.34                      &                      & 6.04                     & 10.22                    & 10.40                     \\
$m = N^{0.65}$ &                      & 0.66                     & 0.80                     & 1.08                      &                      & 2.64                     & 4.32                     & 5.16                      &                      & 5.84                     & 9.40                     & 9.62                      \\
$m = N^{0.70}$ & \multicolumn{1}{l}{} & 0.68                     & 0.86                     & 1.06                      &                      & 3.06                     & 4.08                     & 4.96                      &                      & 6.14                     & 8.68                     & 9.74                      \\
$m = N^{0.75}$ &                      & 0.58                     & 0.90                     & 1.38                      &                      & 3.16                     & 4.72                     & 5.18                      &                      & 6.74                     & 9.36                     & 10.06                     \\ \hline
               &                      & \multicolumn{11}{c}{$\Delta = 0.30, \theta = 0.50$}                                                                                                                                                                                                                                               \\ \cline{3-13}
$m = N^{0.60}$ &                      & 0.62                     & 0.58                     & 0.66                      &                      & 2.68                     & 3.32                     & 3.66                      &                      & 5.28                     & 7.44                     & 7.46                      \\
$m = N^{0.65}$ &                      & 0.52                     & 0.50                     & 0.58                      &                      & 2.50                     & 3.52                     & 3.36                      &                      & 5.18                     & 6.88                     & 7.16                      \\
$m = N^{0.70}$ & \multicolumn{1}{l}{} & 0.62                     & 0.54                     & 0.72                      &                      & 2.62                     & 3.28                     & 3.50                      &                      & 5.72                     & 6.84                     & 7.64                      \\
$m = N^{0.75}$ &                      & 0.58                     & 0.50                     & 0.92                      &                      & 2.88                     & 3.82                     & 4.10                      &                      & 5.72                     & 7.74                     & 8.12                      \\ \hline
               &                      & \multicolumn{11}{c}{$\Delta = 0.30, \theta = 0.75$}                                                                                                                                                                                                                                               \\ \cline{3-13}
$m = N^{0.60}$ &                      & 0.78                     & 1.06                     & 1.20                      &                      & 3.06                     & 5.20                     & 5.68                      &                      & 6.04                     & 9.98                     & 10.54                     \\
$m = N^{0.65}$ &                      & 0.60                     & 0.82                     & 0.90                      &                      & 3.00                     & 4.74                     & 5.10                      &                      & 5.98                     & 8.90                     & 9.90                      \\
$m = N^{0.70}$ & \multicolumn{1}{l}{} & 0.72                     & 0.80                     & 1.00                      &                      & 3.24                     & 4.56                     & 4.78                      &                      & 6.66                     & 9.22                     & 9.58                      \\
$m = N^{0.75}$ &                      & 0.68                     & 0.78                     & 1.20                      &                      & 3.22                     & 4.76                     & 4.90                      &                      & 6.48                     & 9.58                     & 10.00                     \\ \hline
               &                      & \multicolumn{11}{c}{$\Delta = 0.60, \theta = 0.25$}                                                                                                                                                                                                                                               \\ \cline{3-13}
$m = N^{0.60}$ &                      & 1.12                     & 1.78                     & 1.64                      &                      & 4.94                     & 6.88                     & 5.78                      &                      & 9.44                     & 11.78                    & 11.26                     \\
$m = N^{0.65}$ &                      & 1.04                     & 1.70                     & 1.46                      &                      & 4.66                     & 6.30                     & 5.72                      &                      & 8.64                     & 11.50                    & 10.52                     \\
$m = N^{0.70}$ & \multicolumn{1}{l}{} & 1.00                     & 1.48                     & 1.44                      &                      & 4.62                     & 5.54                     & 5.44                      &                      & 9.26                     & 10.58                    & 10.34                     \\
$m = N^{0.75}$ &                      & 1.00                     & 1.46                     & 1.44                      &                      & 4.38                     & 5.98                     & 5.74                      &                      & 9.14                     & 11.00                    & 10.88                     \\ \hline
               &                      & \multicolumn{11}{c}{$\Delta = 0.60, \theta = 0.5$}                                                                                                                                                                                                                                                \\ \cline{3-13}
$m = N^{0.60}$ &                      & 0.72                     & 0.66                     & 0.82                      &                      & 3.18                     & 3.74                     & 3.94                      &                      & 6.44                     & 8.18                     & 8.12                      \\
$m = N^{0.65}$ &                      & 0.60                     & 0.54                     & 0.62                      &                      & 3.04                     & 3.94                     & 3.68                      &                      & 6.34                     & 7.78                     & 7.76                      \\
$m = N^{0.70}$ & \multicolumn{1}{l}{} & 0.78                     & 0.54                     & 0.76                      &                      & 3.20                     & 3.74                     & 3.80                      &                      & 6.78                     & 7.66                     & 8.08                      \\
$m = N^{0.75}$ &                      & 0.76                     & 0.64                     & 0.94                      &                      & 3.36                     & 4.34                     & 4.44                      &                      & 6.92                     & 8.40                     & 8.80                      \\ \hline
               &                      & \multicolumn{11}{c}{$\Delta = 0.60, \theta = 0.75$}                                                                                                                                                                                                                                               \\ \cline{3-13}
$m = N^{0.60}$ &                      & 1.24                     & 1.68                     & 1.80                      &                      & 4.80                     & 6.20                     & 6.24                      &                      & 8.96                     & 11.74                    & 11.20                     \\
$m = N^{0.65}$ &                      & 0.82                     & 1.42                     & 1.38                      &                      & 4.56                     & 6.40                     & 5.66                      &                      & 8.88                     & 11.20                    & 10.80                     \\
$m = N^{0.70}$ & \multicolumn{1}{l}{} & 1.04                     & 1.42                     & 1.54                      &                      & 4.74                     & 5.78                     & 5.66                      &                      & 9.20                     & 11.28                    & 10.16                     \\
$m = N^{0.75}$ &                      & 1.06                     & 1.46                     & 1.72                      &                      & 4.40                     & 6.14                     & 5.84                      &                      & 8.96                     & 11.16                    & 10.68                     \\ \hline
\end{tabular}}\end{table}

\end{document}